\documentclass[11pt,reqno,tbtags]{amsart}
\pdfoutput=1
\usepackage[]{amsmath,amssymb,amsfonts,latexsym,amsthm,enumerate}
\usepackage{amssymb,bbm,mathrsfs}
\usepackage{enumerate,graphicx,paralist}
\usepackage[numeric,initials,nobysame]{amsrefs}

\numberwithin{equation}{section}
\usepackage[colorlinks=true, pdfstartview=FitV, linkcolor=blue, citecolor=blue, urlcolor=blue,pagebackref=false]{hyperref}
\usepackage[margin=11pt
]{caption}

\newtheorem{maintheorem}{Theorem}
\newtheorem{maincoro}[maintheorem]{Corollary}
\newtheorem{mainprop}[maintheorem]{Proposition}

\newtheorem{theorem}{Theorem}[section]
\newtheorem*{theorem*}{Theorem}
\newtheorem{lemma}[theorem]{Lemma}

\newtheorem{corollary}[theorem]{Corollary}

\theoremstyle{definition}{

\newtheorem{definition}[theorem]{Definition}
\newtheorem*{definition*}{Definition}

\newtheorem*{example*}{Example}
}

\theoremstyle{remark}{

\newtheorem*{remark*}{Remark}

}

\newcommand{\R}{\mathbb R}

\newcommand{\Z}{\mathbb Z}
\newcommand{\HS}{\mathbb H}
\newcommand{\deq}{\stackrel{\scriptscriptstyle\triangle}{=}}

\newcommand{\E}{\mathbb{E}}
\renewcommand{\P}{\mathbb{P}}
\DeclareMathOperator{\var}{Var}

\newcommand{\gap}{\text{\tt{gap}}}
\newcommand{\sob}{\alpha_{\text{\tt{s}}}}
\newcommand{\tmix}{t_\textsc{mix}}

\newcommand{\tv}{{\textsc{tv}}}
\newcommand{\Po}{\operatorname{Po}}

\newcommand{\given}{\, \big| \,}

\newcommand{\diam}{\operatorname{diam}}
\newcommand{\one}{\mathbbm{1}}

\renewcommand{\epsilon}{\varepsilon}
\renewcommand{\phi}{\varphi}

\newcommand{\cG}{\mathcal{G}}

\newcommand{\cN}{\mathcal{N}}

\DeclareMathOperator{\dist}{dist}

\newcommand{\block}{{\mathcal{T}}}

\newcommand{\sobinf}{\widehat{\alpha}_{\text{\tt{s}}}}
\newcommand{\ltwo}{{\mathfrak{M}}}
\newcommand{\lltwo}{{\mathfrak{m}}}
\newcommand{\sparse}{\mathcal{S}}
\newcommand{\tX}{\tilde{X}}
\newcommand{\oned}{1\textsc{d} }
\newcommand{\twod}{2\textsc{d} }
\newcommand{\sX}{\mathscr{X}}
\newcommand{\sY}{\mathscr{Y}}
\newcommand{\spinset}{\Sigma}
\newcommand{\Erf}{\operatorname{Erf}}

\date{}

\begin{document}
\title[Cutoff for general spin systems]{Cutoff for general spin systems with arbitrary boundary conditions}

\author{Eyal Lubetzky}
\address{Eyal Lubetzky\hfill\break
Microsoft Research\\
One Microsoft Way\\
Redmond, WA 98052-6399, USA.}
\email{eyal@microsoft.com}
\urladdr{}

\author{Allan Sly}
\address{Allan Sly\hfill\break
Department of Statistics\\
UC Berkeley\\
Berkeley, CA 94720, USA.}
\email{sly@stat.berkeley.edu}
\urladdr{}

\begin{abstract}
The cutoff phenomenon describes a sharp transition in the convergence of a Markov chain to equilibrium.
In recent work, the authors established cutoff and its location for the stochastic Ising model on the $d$-dimensional torus $(\mathbb{Z}/n\mathbb{Z})^d$ for any $d\geq 1$.
The proof used the symmetric structure of the torus and monotonicity in an essential way.

Here we enhance the framework and extend it to general geometries, boundary conditions and external fields to derive a cutoff criterion
that involves the growth rate of balls and the log-Sobolev constant of the Glauber dynamics.
In particular, we show there is cutoff for stochastic Ising on any sequence of bounded-degree graphs with sub-exponential growth under arbitrary external fields provided the inverse log-Sobolev constant is bounded.
For lattices with homogenous boundary, such as all-plus, we identify the cutoff location explicitly in terms of spectral gaps of infinite-volume dynamics on half-plane intersections.
Analogous results establishing cutoff are obtained for non-monotone spin-systems at high temperatures,
including the gas hard-core model, the Potts model, the anti-ferromagnetic Potts model and the coloring model.
\end{abstract}

\maketitle


\section{Introduction}

The total-variation \emph{cutoff phenomenon} describes a sharp transition in the $L^1$-distance of a finite Markov chain from equilibrium, dropping abruptly from near its maximum to near $0$ (see formal definitions in~\S\ref{sec:prelim-cutoff}). Since its discovery in the early 80's by Aldous and Diaconis (\cites{AD} following~\cites{Aldous,DiSh}) the cutoff phenomenon has been believed to be widespread in Markov chains. Until very recently, however, rigorous proofs of cutoff were confined to very few cases where the stationary measure was well-understood and fairly simple, e.g.\ uniform, 1-dimensional unimodal etc. The focus of this work is on cutoff for Glauber dynamics for spin-systems such as the Ising model, 
chains which are natural models for the evolution of these systems in addition to being their most practiced sampling methods.
Despite extensive study, the understanding of the stationary measure in these models remains limited. For instance, for the Ising model on $\mathbb{Z}^3$ the value of the critical temperature is unknown, as is the basic question of whether spin-spin correlations at criticality decay with distance.

Peres conjectured in 2004 that Glauber dynamics for the Ising model on any sequence of bounded-degree transitive graphs would exhibit cutoff at high temperatures, yet even for the \oned torus (Ising on a cycle) this was open until recently. In the companion paper~\cite{LS1} the authors proved cutoff and established its location for Glauber dynamics for the Ising model on the $d$-dimensional torus $(\mathbb{Z}/n\mathbb{Z})^d$ for any $d\geq 1$ at any temperature where the stationary measure satisfies the \emph{strong spatial mixing} condition which requires that the effect of changes in boundary conditions decay exponentially in distance  (this holds on $\Z$ at all temperatures and on $\Z^2$ all the way up to the critical temperature; see \S\ref{sec:prelim} for definitions and background.) The proof featured a framework for eliminating the dependencies between distant small boxes, with which the $L^1$-mixing of the whole system could be reduced to $L^2$-mixing of its projections onto these boxes. The symmetry of the torus, giving any such box the same effect on mixing,  played a key role in the analysis
(e.g., the basic setting of a box in $\mathbb{Z}^d$ with free boundary conditions was not covered), as did the monotonicity of the Ising model.

In this work we enhance the above framework to forsake the symmetry and monotonicity limitations. Postponing standard definitions to~\S\ref{sec:prelim}, we next describe the new results, first for the Ising model and thereafter for non-monotone systems such as the gas hard-core model and the Potts model.

{\setlength{\parskip}{-0.05cm}

\subsection{Results}
Theorem~\ref{mainthm-arbitrary} below formulates a cutoff criterion for stochastic Ising on general graphs with arbitrary boundary conditions, interaction strengths and external fields (possibly non-uniform). The criterion factors in the growth rate of balls of logarithmic radius and the log-Sobolev constants of the induced system on subgraphs of poly-logarithmic diameter. As one application we obtain that on bounded-degree graphs where these log-Sobolev constants are bounded away from 0 (e.g., lattices in the strong spatial mixing regime) there is cutoff as long as the graphs have \emph{sub-exponential growth}.

\begin{maintheorem}\label{mainthm-arbitrary}
Let $G$ be a connected graph on $n$ vertices with maximal degree $\Delta \geq 2$, and
consider Glauber dynamics for the {\sf ferromagnetic Ising model} on $G$ with arbitrary (possibly non-uniform) interactions and external fields.
Let
\begin{align*}
\begin{array}{rcl}
\sobinf &=& \min\left\{ \sob(H) \;:\; H\subset G \mbox{ with }\diam(H) \leq \log^2n \right\} \,,\\
\noalign{\smallskip}
\rho &=& \max\left\{ \left|B_v\left(\lfloor 10 \Delta \sobinf^{-1}\log n\rfloor\right)\right| \;:\; v\in V(G)\right\} \,,
\end{array}
\end{align*}
where $\sob(H)$ is the log-Sobolev constant of the dynamics on $H$ with free boundary and the original interactions and external field, and $B_v(r)$ is the subgraph induced on the ball of radius $r$ around the vertex $v$.
Then
 \[\tmix(\epsilon)-\tmix(1-\epsilon)\leq 16 \Delta \sobinf^{-2}\log\rho \]
 for any $\epsilon>0$ and large enough $n$.
 In particular, for any sequence $G_n$ with $\Delta\sobinf^{-1} = O(1)$ and $\rho = n^{o(1)}$ the dynamics has cutoff,
 as $\tmix(1/e) \gtrsim \log n$.
\end{maintheorem}
}

The above theorem allowed for an arbitrary (possibly non-uniform) external field, which in particular encompasses any arbitrary boundary condition. As mentioned above, it is known (due to~\cite{MO}) that the log-Sobolev constant of stochastic Ising on a finite box in $\mathbb{Z}^d$ is bounded away from 0 whenever there is strong spatial mixing. The growth rate parameter $\rho$ defined in Theorem~\ref{mainthm-arbitrary} there satisfies $\rho \asymp \log n $, whence we deduce the following:

\begin{maincoro}\label{maincor-bc}
Let $d\geq 1$ and consider Glauber dynamics for the {\sf Ising model} on a box $\Lambda \subset \Z^d$ of side-length $n$ under
arbitrary boundary conditions, local interactions and external fields. Then throughout the regime of strong spatial mixing the dynamics exhibits cutoff
with a window of $O(\log\log n)$.
\end{maincoro}

For instance, the above corollary implies that on $\mathbb{Z}^2$ with free boundary conditions
the dynamics exhibits cutoff for any $\beta < \beta_c =\frac12\log(1+\sqrt{2})$.

The proofs of the aforementioned results are not specific to the Ising model but rather can be applied to Glauber dynamics for any spin system model at high enough temperature, notably including non-monotone systems such as the Potts model. This is formalized in the generic Theorem~\ref{t:general} (see \S\ref{sec:general}), from which we can derive the following corollaries:
\begin{maintheorem}\label{mainthm-potts-Zd}
Let $d\geq 1$ and consider the {\sf Potts model} on a box $\Lambda \subset \Z^d$ of side-length $n$
with $q \geq 2$ colors, inverse-temperature $\beta \geq 0$ satisfying $\beta < \frac12 \log\big(\max\big\{1+\frac{1}{2d-1},(1+\frac{q}{2d})^{1/2d}\big\}\big)$ and arbitrary boundary conditions. Then the Glauber dynamics exhibits cutoff with a window of $O(\log\log n)$.

Furthermore, the analogous statement also holds for Glauber dynamics for the {\sf anti-ferromagnetic Potts model} as long as
$ 0 \geq \beta \geq -\frac1{2}\log(1+\frac{1}{2dq-1})$.
\end{maintheorem}
\begin{maintheorem}\label{mainthm-hardcore-Zd}
Let $d\geq 1$ and consider the {\sf gas hard-core model} on a box $\Lambda \subset\Z^d$ of side-length $n$ with fugacity $\lambda < \frac1{2d-2}(1+\frac1{2d-2})^{2d-1}$ and arbitrary boundary conditions. The heat-bath dynamics has cutoff with window $O(\log\log n)$.
\end{maintheorem}
\begin{maintheorem}\label{mainthm-coloring-Zd}
Let $d\geq 1$ and consider the {\sf proper coloring model} on a box $\Lambda \subset \Z^d$ of side-length $n$
with $q \geq 4d(d+1)$ colors and arbitrary boundary conditions. Then the heat-bath dynamics has cutoff with window $O(\log\log n)$.
\end{maintheorem}

\begin{remark*}
While the above results were stated for $\Z^d$, analogous results hold for any $d$-dimensional lattice $\mathbb{L}^d$ with finite periodicity, such as triangular, hexagonal etc., in particular including non-bipartite lattices
(in which one cannot use the standard monotone reordering procedure of an anti-monotone system).
Other notable examples for the lattice $\mathbb{L}^d$ include:
\begin{compactenum}[(i)]
  \item \label{it-long-range} a $d$-dimensional lattice (e.g.\ $\Z^d$) with long range interactions (edges between any two vertices with distance at most $l$ for some $l \geq 1$ fixed).
  \item \label{it-prod} a product of a $d$-dimensional lattice with any fixed graph $H$.
\end{compactenum}
Similarly, Theorems~\ref{mainthm-potts-Zd}--\ref{mainthm-coloring-Zd} were stated for the heat-bath dynamics yet the proofs hold also for Metropolis-Hastings with suitably modified constants.
\end{remark*}

In all the results listed thus far there was no closed form for the cutoff location.
The proofs specify this location as a threshold for the cumulative $L^2$-mixing on poly-logarithmic balls around each vertex, in which the effect of different subgraphs can greatly vary in general. When the boundary conditions are homogenous (e.g., free or all-plus) we can obtain an explicit formula for the cutoff location in terms of spectral gaps of the infinite-volume dynamics
on half-plane intersections. This is stated next for \twod lattices with all-plus boundary. (An analogous statement holds for free/minus boundary.)

\begin{maintheorem}\label{mainthm-plusBC2d}
Consider Glauber dynamics for the {\sf Ising model} on $\Lambda\subset \Z^2$, an $n\times n$ box with
all-plus boundary conditions. Throughout the high temperature regime $\beta < \beta_c = \frac12\log(1+\sqrt{2})$
cutoff occurs at $(\lambda_\infty \,\wedge\, 2\lambda_{\mathbbm{H}})^{-1}\log n$,
where $\lambda_\infty$ and $\lambda_{\mathbbm{H}}$ are the spectral gaps of the dynamics on the
infinite-volume lattice and the half-plane with all-plus boundary condition, respectively.
\end{maintheorem}

In \S\ref{sec:cutoff-locations} we prove the generalization of the above theorem (Theorem~\ref{t:plusGeneral}) to $\Z^d$ for any $d$. In that generality, the cutoff location is given in terms of $d$ separate infinite-volume spectral gaps in $\Z^d$ where the corresponding boundary is imposed on the half-planes in a subset of the $d$ coordinates.

Furthermore, as in the remark following Theorems~\ref{mainthm-potts-Zd}--\ref{mainthm-coloring-Zd}, the arguments extend
beyond $\Z^d$ to give the cutoff location for any $d$-dimensional lattice $\mathbb{L}^d$ with finite periodicity and homogenous boundary conditions.
For instance, a special case of example~\eqref{it-prod} from that remark is the \emph{circular ladder} graph, the product of the cycle $\Z/n\Z$ with an edge. In~\cite{LPW}*{Theorem~15.10} it was shown that Glauber dynamics for the Ising model on this graph mixes in continuous-time $O(n)$ and noted that the correct order would be $O(\log n)$ via a block-dynamics argument. Our results imply that in fact there is cutoff in this setting and the mixing asymptotics are $(2\lambda_\infty)^{-1}\log n$ where $\lambda_\infty$ is the (explicitly known) spectral gap of Glauber dynamics on the infinite ladder. 

Contrary to the results establishing cutoff, the argument that relates its location to infinite-volume spectral-gaps does rely on the monotonicity of the Ising model in an essential way. Nevertheless, in the special case of bipartite lattices such as $\Z^d$
  the standard monotone reordering of the partial order on lattice configurations extends the arguments also to anti-monotone systems. The analogue of Theorem~\ref{mainthm-plusBC2d} (cutoff and its location in terms of spectral-gaps of infinite-volume dynamics on half-plane intersections) thus holds also for
\begin{compactenum}[(i)]
  \item {\sf anti-ferromagnetic Ising model} on $\Z^d$ under all-plus/all-minus/free b.c.
  \item {\sf gas hard-core model} on $\Z^d$ under all-plus/all-minus/free b.c.
\end{compactenum}
and all throughout the strong spatial mixing regime.


\begin{figure}
\includegraphics[width=5.1in]{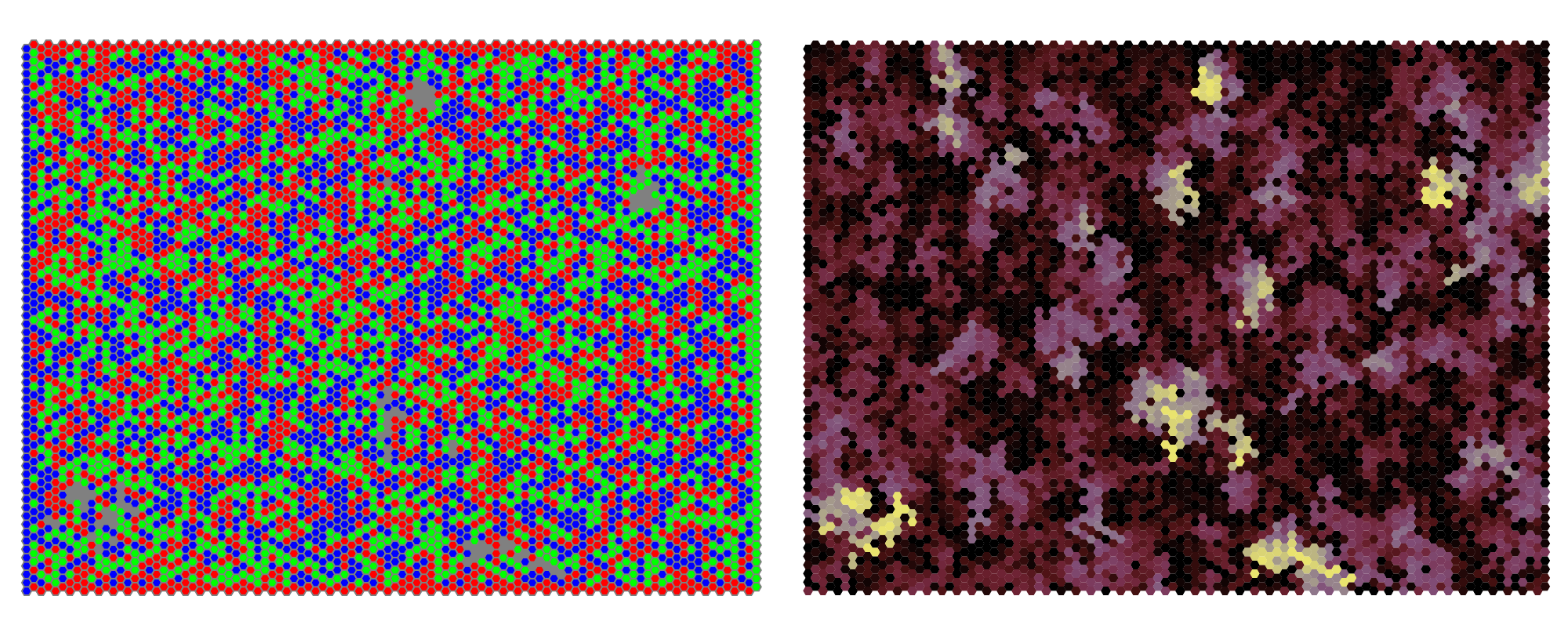}
\vspace{-0.35in}
\caption{Potts model with $q=3$ colors on a $64\times100$ hexagonal lattice at high temperature ($\beta=0.1$).
On right, update support for the barrier dynamics in yellow, intensity of red depicts the age outside the support. On left, the coupled configuration wherever uniquely determined.}
\vspace{-0.05in}
\label{fig:support-potts}
\end{figure}

\subsection{Methods}
At high temperatures the fast decay of spatial dependence implies that the measure rapidly becomes well mixed locally.  Moreover, propagation of information occurs at a constant rate and so the configuration in distant parts of the graph will be close to independent.  This suggests that a product Markov chain is the right heuristic to understand the convergence of high temperature Glauber dynamics. We next discuss showing cutoff for such chains, followed by our approach to reduce the problem to that setting.

\subsubsection{An $L^1$ to $L^2$ reduction for product chains}
An important step in our proof is bounding the total variation distance from stationarity of product chains by the $L^2$-distance of its component chains.  In general for $1\leq p \leq \infty$, the $L^p(\pi)$ distance between the measures $\nu$ and $\pi$ that nowhere vanishes is 
\[
\left\|\nu - \pi \right\|_{L^p(\pi)} = \Big(\sum_x \Big|\frac{\nu(x)}{\pi(x)} -1 \Big|^p\pi(x)\Big)^{1/p}\,.
\]
The mentioned reduction is formalized by the next proposition (proved in~\S\ref{sec-l1-l2}) which we believe is of independent interest.

\begin{mainprop}\label{prop-l1-l2}
Let $X_t = (X_t^1,\ldots,X_t^n)$ be a product chain, i.e.\ the $X_t^i$'s are mutually independent ergodic chains with stationary measures $\pi_1,\ldots,\pi_n$ respectively. Let $\pi$ denote the product measure of $\pi_{1},\ldots,\pi_{n}$ and define
\begin{align}
  \label{eq-product-m-def}
  \ltwo_t = \sum_{i=1}^n \left\|\P(X_t^i\in\cdot)-\pi_i \right\|^2_{L^2(\pi_i)}\,.
\end{align}
Then $\left\|\P(X_t\in\cdot)-\pi\right\|_\tv \leq \sqrt{\ltwo_t}$ for any $t>0$ and furthermore,
for every $\delta>0$ there exists some $\epsilon>0$ so that the following holds: If
for some $t>0$
\begin{align}
  \label{eq-product-assumption}
 \max_{i} \left\| \P(X_t^i\in \cdot)-\pi_i\right\|_{L^\infty(\pi_i)} < \epsilon
\end{align}
then
\begin{align}
  \label{eq-product-limit}
  \left|\left\|\P(X_t\in\cdot)-\pi\right\|_\tv - \big(2\Phi(\sqrt{\ltwo_t}/2)-1\big)\right| < \delta\,,
\end{align}
where $\Phi$ denotes the cumulative distribution function of the standard normal.
In particular, if for a family of chains $\ltwo_t\to 0$ then
$\left\|\P(X_t\in\cdot)-\pi\right\|_\tv \to 0$ whereas if
$\ltwo_t \to \infty$ and~\eqref{eq-product-assumption} holds for $\epsilon\to 0$ then
$\left\|\P(X_t\in\cdot)-\pi\right\|_\tv \to 1$.
\end{mainprop}

This identifies the cutoff location for product chains as above as the time $t$ such that $\ltwo_t \asymp 1$, and further establishes the total-variation distance from equilibrium within the cutoff window to be $2\Phi(\sqrt{\ltwo_t}/2)-1 = \Erf(\sqrt{\ltwo_t/8})$ asymptotically, where $\Erf(x)=\frac{2}{\sqrt{\pi}}\int_0^x e^{-t^2}dt$ is the error-function.

\begin{example*} Let $(X_t)$ be continuous-time random walk on the hypercube $\Z_2^n$. That is, its discrete-time analogue $(K_t)$ flips a uniform coordinate at each step, while $X_t = K_{N_t}$ for a Poisson random variable $N_t\sim \Po(t)$ (eliminating periodicity issues). This was one of the original examples of cutoff, which occurs here at time $\frac14 n \log n$
as shown by Aldous~\cite{Aldous}. Moreover, cutoff occurs within a window of $O(n)$, as follows from explicit bounds of~\cite{DiSh2} on the distance of the lazy discrete-time chain (applies $K$ with probability $\frac{n}{n+1}$ and is idle otherwise) from equilibrium at time $t=\frac14 n\log n + c n$. These bounds were refined in~\cite{DGM}*{Theorem 1} to show this distance is $\Erf(e^{-2c}/\sqrt{8})+o(1)$.

To reobtain this result for the continuous-time chain via Proposition~\ref{prop-l1-l2} argue as follows.
Let $(\bar{X}_t)$ be the lazy continuous-time chain corresponding to the discrete kernel $\bar{K}=(I+K)/2$. It is well known (and easy to see) that the heat-kernels corresponding to $(X_t),(\bar{X}_t)$ satisfy $H_t = \bar{H}_{2t}$,
  thus it suffices to estimate $|\P(\bar{X}_t\in\cdot)-\pi\|_\tv$. Yet $(\bar{X}_t)$ is the product chain of i.i.d.\ chains $(\bar{X}_t^i)$ with stationary measures that are uniform on $\{0,1\}$, each of which is flipped to a uniform state at rate $1/n$. One can easily verify that in this case $ \left\|\P(\bar{X}_t^i\in\cdot)-\pi_i \right\|^2_{L^2(\pi_i)} = \mathrm{e}^{-2t/n}$. As such, the proposition implies that $(\bar{X}_t)$ has cutoff at $\frac12 n \log n$. Furthermore, if $t = \frac12n\log n + c n$ for some $c \in \R$ then $|\P(\bar{X}_t\in\cdot)-\pi\|_\tv = \Erf(e^{-c}/\sqrt{8}) + o(1)$. Translating this to $(X_t)$ recovers the mentioned result that if $t=\frac14 n\log n + c n$ for some $c\in \R$ then
 \[\left|\P(X_t\in\cdot)-\pi\right\|_\tv = \Erf\big(e^{-2c}/\sqrt{8}\big) + o(1)\,.\]
\end{example*}


\begin{figure}
\centering
\includegraphics[width=5in]{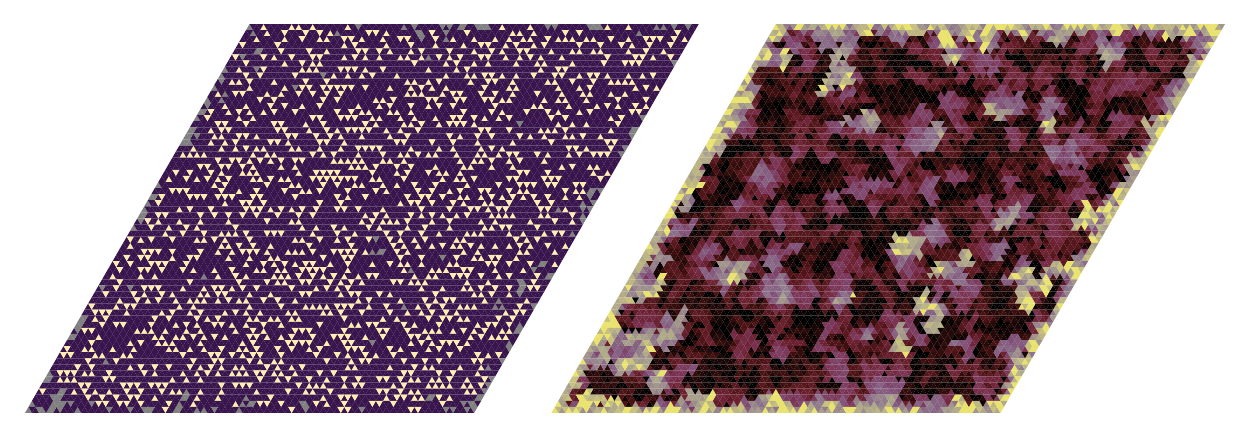}
\caption{
Gas hard-core model on a $64\times128$ triangular lattice at low fugacity ($\lambda=0.6$).
The update support is highlighted on the right vs.\ the coupled configuration wherever uniquely determined on the left.}
\label{fig:support-hc}
\end{figure}

\begin{remark*}
Proposition~\ref{prop-l1-l2} in fact applies to any two product-measures (without requiring that one arises from a Markov process and the other is its stationary measure) and provides an $L^1$ to $L^2$ reduction of the distance between them.
  One virtue of formulating it in terms of Markov chains (beyond its application in this work)
  is that one can then derive
  the $L^\infty$ condition on $X_t^i$ from the analogous $L^2$ requirement on $X_{t/2}^i$ (see, e.g., Corollary~\ref{cor-products}).
\end{remark*}

\subsubsection{Breaking dependencies}
While the product chain paradigm is useful in explaining the presence of cutoff, the key ideas of the proof are methods to break the dependencies inherent in the dynamics.  Proposition~\ref{prop-l1-l2} does suggest a natural way to establish lower bounds on the mixing time.  If one takes well separated blocks $B_i \subset V$, the projections of the chain  onto the blocks are essentially independent.  Taking $X_t^i=X_t(B_i)$ in Proposition~\ref{prop-l1-l2} gives lower bounds on the distance from stationarity which are sufficiently strong to give sharp enough lower bounds on the mixing time.

The dependencies of the measure and the chain, however, make sharp upper bounds far more challenging.
Since projection decreases the total variation distance this approach does not yield upper bounds.  Despite the rapid decay of correlations with distance, it remained a challenging open problem to give sharp total variation distance bounds even for the $\oned$ Ising model.
To overcome this in the case of the Ising model on the torus in~\cite{LS1} we introduced a method to bound the total variation distance at time $t+s$ by the expected total variation distance at time $s$ of the chain projected onto a random set we call the \emph{update support}.

Viewing the Markov chain from time $t$ to time $t+s$ as a random function $f_W:\{+1,-1\}^V \to \{+1,-1\}^V$ where $f_W(X_t)=X_{t+s}$ it may be the case that the value of $f_W$ is determined by only a subset of the spins of $V$ --- the smallest such set is the update support.  Since by definition the spins outside of this set have no effect on the value of the chain at time $t+s$, the total variation distance at time $t+s$ is bounded by the expected projection onto the update support at time $t$. (Lemma~\ref{lem-support} places this in a more general context of random mapping representations for Markov chains.)

A key step in our analysis is then to analyze the structure of the update support.  While~\cite{LS1} used the symmetries of the torus in numerous ways, in this paper we consider graphs and systems with highly inhomogeneous structure and external fields.  The hypothesis of Theorem~\ref{mainthm-arbitrary} is chosen in such a way that the dynamics in most local neighbourhoods of radius $r \asymp \Delta \sobinf^{-1} \log n$ would couple completely in time of smaller order than the mixing time.  The sub-exponential growth is essential as it ensures that these neighborhoods have size $n^{o(1)}$.

To ensure that well separated vertices are independent we modify the map $f_W$ to satisfy this property yet in such a way that it can be coupled to the true dynamics with high probability. We call this modified process the \emph{barrier dynamics}, constructed so that information cannot travel further than a given distance.
Using this independence we wish to show that when $s$ is sufficiently large, yet still of smaller order than the mixing time, the update support is \emph{sparse}. Here again the assumption of sub-exponential growth is essential as it allows us to take union bounds and show that the support size decays exponentially and that it splits into small well-separated components of diameter at most $\log^2 n$ w.h.p.\ (see Figures~\ref{fig:support-potts}--\ref{fig:support-ising}).  Finally, our requirement on the log-Sobolev constant is chosen in such a way that the dynamics restricted to these regions mixes rapidly.

\begin{figure}
\centering
\includegraphics[width=5in]{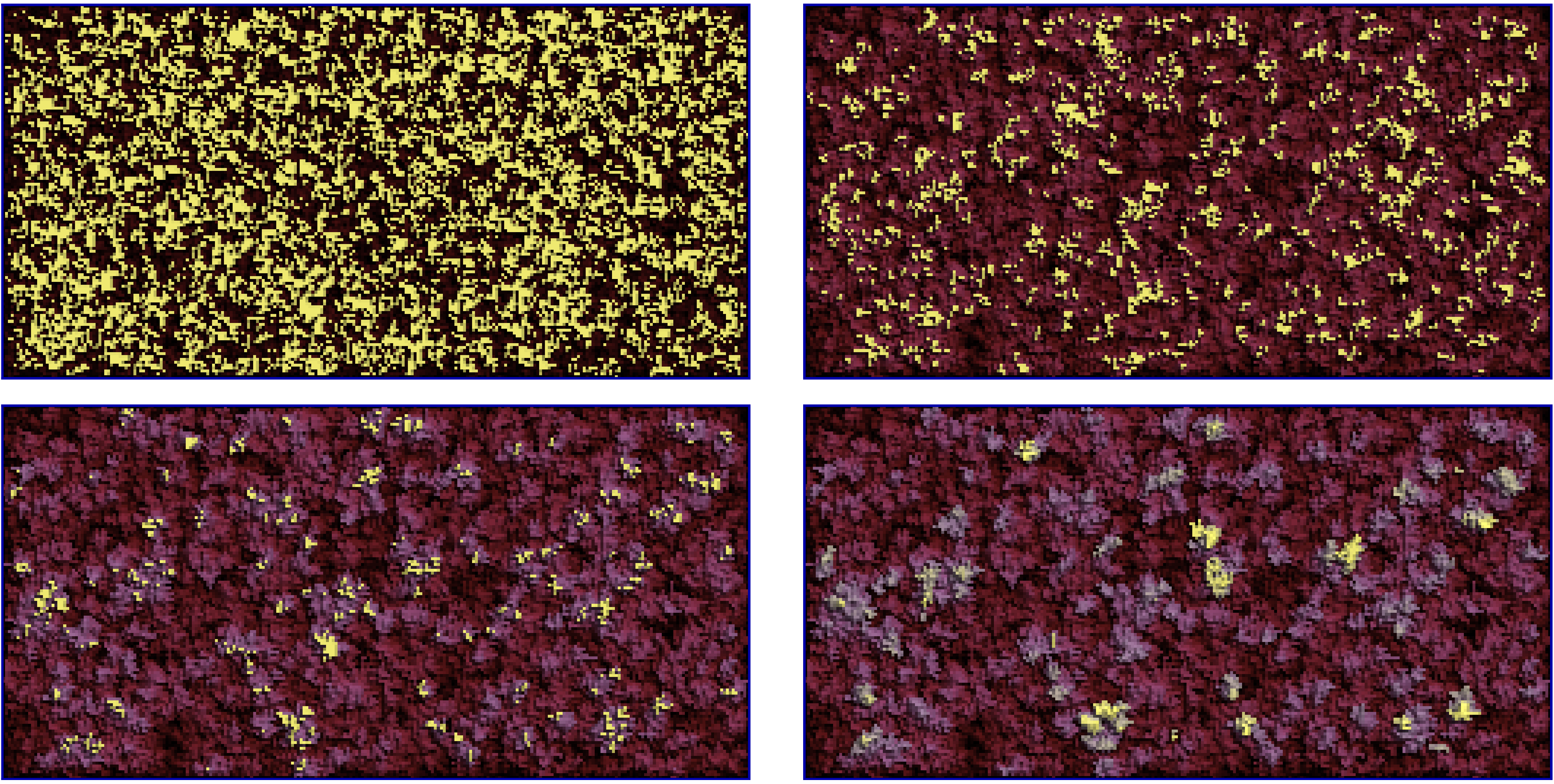}
\caption{Ising model on a $128\times 256$ square lattice at $\beta=0.32$.
Frames depict evolution of the update support of the barrier dynamics (support sites in yellow) into microscopic disconnected components.}
\label{fig:support-ising}
\end{figure}



The quantitative bounds of Proposition~\ref{prop-l1-l2} then allow us to compare the derived upper and lower bounds. Again, arbitrary geometry and external fields may grant much stronger effect to certain parts of the system on mixing compared to others (whereas the setting of the torus studied in~\cite{LS1} had complete symmetry and the problem reduced to a product of i.i.d.\ chains). Nevertheless, it is possible to show that there exists a well separated subset of blocks which dominates an expected sparse set of blocks from the matching upper bound.  Finding this set of blocks is done in a greedy manner.  Overall, we identify a time $t^\star$ at which the dynamics is far from mixed through the lower bound and then is well mixed at time $t^\star + s $ through the upper bound.

\subsubsection{General spin-systems}
In \S\ref{sec:general} we extend our analysis to the case of general spin-systems.  Our construction of the function $f_W$ described above implicitly used a grand-coupling of the chain starting from all possible initial configurations.  In the case of the Ising model this is the standard monotone coupling, through which the chains from all starting states coalesce rapidly (namely, exponentially fast) at high temperatures.
For general models, e.g., non-monotone systems, we extend the framework to remain valid so long as such a grand coupling is available (see Theorem~\ref{t:general}). For concreteness we give a condition, valid for high enough temperatures, under which one can construct a grand coupling as required, thus implying cutoff and the aforementioned Theorems~\ref{mainthm-potts-Zd}--\ref{mainthm-coloring-Zd}. This approach is akin to methods used in Coupling From The Past techniques for perfect simulation.

\subsubsection{Cutoff location in terms of infinite-volume spectral-gaps}
In \S\ref{sec:cutoff-locations} (see Theorem~\ref{t:plusGeneral}) we derive the cutoff location for the Ising model on a box in $\Z^d$ under free/all-plus/all-minus boundary conditions. Unlike the setting of the torus considered in~\cite{LS1}, the presence of boundary conditions destroys the symmetry of the graph and creates a disparity between the mixing effects of vertices near the boundary vs.\ those near the center of the box.  As such, the cutoff location becomes a function of the spectral gap of the dynamics not just on the infinite-volume lattice but also on intersections of half-spaces. Our approach is to classify vertices according to their position in the box in terms of the number of faces they are close to and to couple the local dynamics to the one in the appropriate half-space. We note in passing that at present it is unknown whether for instance the spectral-gaps featured in Theorem~\ref{mainthm-plusBC2d} satisfy $\lambda_{\mathbbm{H}} < \lambda_\infty$. (A weak inequality holds, while equality would imply that the cutoff location specified there is only a function of $\lambda_\infty$.)
There are plausible heuristics suggesting that these half-space spectral-gap terms are imperative in at least some settings, e.g.,
models with stronger correlations on the boundary than in the bulk due to boundary conditions (see~\cite{Martinelli97} for more details).


\subsubsection{Organization}
In \S\ref{sec:prelim} we give the main definitions related to Markov chains, spin systems and the Glauber dynamics.  We prove Proposition~\ref{prop-l1-l2} in \S\ref{sec-l1-l2}.  In \S\ref{sec:arbitrary} we prove Theorem~\ref{mainthm-arbitrary} showing that cutoff applies in general on graphs of sub-exponential growth whenever the log-Sobolev constants are suitably bounded.  Section~\ref{sec:general} extends our methods for proving cutoff to general (not necessarily monotone) spin systems at high temperatures, establishing Theorems~\ref{mainthm-potts-Zd}, \ref{mainthm-hardcore-Zd} and~\ref{mainthm-coloring-Zd} for the Potts, coloring and hardcore models respectively.  Finally, \S\ref{sec:cutoff-locations} deals with the case of the Ising model with plus or free boundary conditions proving Theorem~\ref{mainthm-plusBC2d}.

\section{Preliminaries}\label{sec:prelim}
\subsection{Mixing and the cutoff phenomenon}\label{sec:prelim-cutoff}
The $L^1$ (or total-variation) distance of a Markov chain from equilibrium is one of the most important notions of convergence in MCMC theory. For two probability measures $\nu_1,\nu_2$ on a finite space $\Omega$ the total-variation distance is defined as
\[
\|\nu_1-\nu_2\|_\tv = \max_{A\subset \Omega} |\nu_1(A)-\nu_2(A)| = \frac12\sum_{\sigma\in\Omega} |\nu_1(\sigma)-\nu_2(\sigma)| \,,
\]
i.e.\ half the $L^1$-distance between the two measures.
Let $(X_t)$ be an ergodic finite Markov chain $(X_t)$ with stationary distribution $\pi$. The total-variation mixing-time of $(X_t)$, denoted $\tmix(\epsilon)$ for $0<\epsilon<1$, is defined to be
\[ \tmix(\epsilon) \deq \inf\left\{t \;:\; \max_{\sigma_0 \in \Omega} \| \P_{\sigma_0}(X_t \in \cdot)- \pi\|_\tv \leq \epsilon \right\}\,,\]
where here and in what follows $\P_{\sigma_0}$ denotes the probability given $X_0=\sigma_0$.
A family of ergodic finite Markov chains $(X_t)$, indexed by an implicit parameter $n$, is said to exhibit \emph{cutoff} iff the following sharp transition in its convergence to stationarity occurs:
\begin{equation}\label{eq-cutoff-def}\lim_{n\to\infty} \frac{\tmix(\epsilon)}{\tmix(1-\epsilon)}=1 \quad\mbox{ for any $0 < \epsilon < 1$}\,.\end{equation}
That is, $\tmix(\alpha)=(1+o(1))\tmix(\beta)$ for any fixed $0<\alpha<\beta<1$. Addressing the asymptotic error in this formulation is the notion of the \emph{cutoff window}: A sequence $w_n = o\big(\tmix(e^{-1})\big)$ is a cutoff window if $\tmix(\epsilon) = \tmix(1-\epsilon) + O(w_n)$ holds for any $0<\epsilon<1$ with an implicit constant that may depend on $\epsilon$.
Equivalently, if $t_n$ and $w_n$ are two sequences such that $w_n =o(t_n)$, one may define that a sequence of chains exhibits cutoff at $t_n$ with window $w_n$ iff
\[\left\{\begin{array}
  {r}\displaystyle{\lim_{\gamma\to\infty} \liminf_{n\to\infty}
 \max_{\sigma_0 \in \Omega} \| \P_{\sigma_0}(X_{t_n-\gamma w_n} \in \cdot)- \pi\|_\tv
  = 1}\,,\\
  \noalign{\medskip}
  \displaystyle{\lim_{\gamma\to \infty} \limsup_{n\to\infty}
 \max_{\sigma_0 \in \Omega} \| \P_{\sigma_0}(X_{t_n+\gamma w_n} \in \cdot)- \pi\|_\tv
  = 0}\,.
\end{array}\right.\]

The cutoff phenomenon was first identified for random transpositions on the symmetric group in \cite{DiSh},
and for the riffle-shuffle and random walks on the hypercube in \cite{Aldous}. Its name was coined
by Aldous and Diaconis in their famous paper~\cite{AD}, where cutoff was shown for the top-in-at-random card shuffling process.
See \cites{Diaconis,CS,SaloffCoste2} and the references therein for more on the cutoff phenomenon.
Note that in the examples above, as well as in most others where cutoff has been rigorously shown, the stationary distribution has
many symmetries or is essentially one-dimensional (e.g.\ uniform on the symmetric group~\cite{DiSh},
uniform on the hypercube~\cite{Aldous} and one-dimensional birth-and-death chains~\cite{DLP2}).

Establishing cutoff can prove to be challenging already for simple families of chains. For instance,
even for random walks on random regular graphs (where the stationary distribution is uniform) cutoff was only recently
verified in~\cite{LS2}. Prior to that work there was no known example of a family of bounded-degree graphs where the random walk exhibits cutoff.

\subsection{Ising model}\label{sec:prelim-ising}

The \emph{Ising model} on a finite graph $G$ with vertex-set $V$ and edge-set $E$
is defined as follows. Its set of possible configurations is $\Omega=\{\pm1\}^V$, where each configuration
corresponds to an assignment of plus/minus spins to the sites in $V$. The probability that the system is in a
configuration $\sigma \in \Omega$ is given by the Gibbs distribution
\begin{equation}
  \label{eq-Ising}
  \mu(\sigma)  = \frac1{Z(\beta)} \exp\left(\beta \sum_{uv\in E} \sigma(u)\sigma(v) + h \sum_{u \in V} \sigma(u)\right) \,,
\end{equation}
where the partition function $Z(\beta)$ is a normalizing constant.
The parameters $\beta$ and $h$ are the inverse-temperature and external field respectively; for $\beta \geq 0$ we say that the
model is ferromagnetic, otherwise it is anti-ferromagnetic.
These definitions extend to infinite locally finite graphs (see e.g.\ \cites{Liggett,Martinelli97}).

In full generality the model associates arbitrary (possibly non-uniform) interaction strengths to the bonds $\{\beta_{uv} : uv\in E\}$ as well as arbitrary external fields to different sites $\{ h_u : u\in V\}$.

We denote the boundary of a set $\Lambda\subset V$ as the neighboring sites of $\Lambda$ in $V\setminus\Lambda$
and call $\tau\in\{\pm1\}^{\partial \Lambda}$ a \emph{boundary condition}. A periodic boundary condition on $(\Z/n\Z)^d$
corresponds to a $d$-dimensional torus of side-length $n$.
For $\Lambda \subset V$, denote by $\sigma(\Lambda)$ the spins that $\sigma$ assigns to $\Lambda$.
Let $\mu_\Lambda^\tau$ denote the measure on configurations on $\Lambda$ given the boundary condition $\tau \in \{\pm1\}^{\partial\Lambda}$, that is,
the conditional measure $\mu\big(\sigma(\Lambda)\in \cdot \mid \sigma(\partial\Lambda)=\tau(\partial\Lambda)\big)$.
We will often refer to the projection of this measure onto a subset of the spins $A \subset \Lambda$ (i.e.\ the marginal of $\mu_\Lambda^\tau$ on $\{\pm1\}^A$) which we will denote by $\mu_\Lambda^\tau|_A$.
The notation $\mu_\Lambda^\emptyset$ will denote the measure on $\{\pm1\}^\Lambda$ under free boundary, i.e.\ the one obtained by setting to $0$ all the interactions between $\Lambda$ and $\partial\Lambda$.

\subsection{Spin-system models}\label{sec:prelim-spin-systems}
%
In \S\ref{sec:general} we extend the cutoff criterion provided in Theorem~\ref{mainthm-arbitrary} for Glauber dynamics for the Ising model to a general class of spin-systems with nearest-neighbor interactions, defined as follows.

Let $G=(V,E)$ be a finite graph and let $\spinset$ be a finite set (spins).
Further consider functions $g_{u,v}:\spinset^2 \to \R\cup\{-\infty\}$ for every bond $(u,v)\in E$ (nearest-neighbor interactions) and functions $h_u : \spinset\to\R$ for all $u\in V$ (external fields).
The corresponding spin system is the probability distribution on the set of configuration $\spinset^V$ given by
\[
\mu(\sigma)=\frac1Z \exp\bigg[ \sum_{(u,v)\in E }g_{u,v}\big(\sigma(u),\sigma(v)\big) + \sum_{u\in V} h_u\big(\sigma(u)\big)\bigg]\,,
\]
where $Z$ is a normalizing constant (the partition function).

The following well-known models were featured in Theorems~\ref{mainthm-potts-Zd}--\ref{mainthm-coloring-Zd}:
\begin{itemize}[\indent$\bullet$]
  \item The $q$-state {\sf Potts model} with inverse-temperature $\beta$:
  \[ \spinset=\{1,\ldots,q\} ~,~~
  g_{u,v}(x,y)=\left\{\begin{array} {cl}2\beta & x=y \\ 0 & \mbox{otherwise} \end{array}\right. ~,~~
  h\equiv 0\,.\]
      (Notice the factor of $2$ in the interactions $g_{u,v}$ whose sole purpose is to make $\beta$ consistent with the Ising model definition~\eqref{eq-Ising} when $q=2$.)
      The model is \emph{ferromagnetic} if $\beta\geq 0$ and \emph{anti-ferromagnetic} if $\beta<0$.
  \item The {\sf proper $q$-colorings model}: the special case of the anti-ferromagnetic Potts model when $\beta = -\infty$.
  \item The {\sf gas hard-core model} with fugacity $\lambda>0$:
  \[ \spinset=\{0,1\}\,,~ g_{u,v}(x,y)=\left\{\begin{array}
    {cl}-\infty & x=y=1 \\ 0 & \mbox{otherwise}
  \end{array}\right.,~ h_u(x) = \one_{\{x=1\}}\log\lambda\,.\]
\end{itemize}

%

\subsection{Glauber dynamics for the Ising model}\label{sec:prelim-glauber}
Glauber dynamics for the Ising model (also dubbed the \emph{Stochastic Ising} model) is a family of continuous-time Markov chains on the state space $\Omega$, reversible with respect to the Gibbs distribution, given by the generator
\begin{equation}
  \label{eq-Glauber-gen}
  (\mathscr{L}f)(\sigma)=\sum_{x\in \Lambda} c(x,\sigma) \left(f(\sigma^x)-f(\sigma)\right)
\end{equation}
where $\sigma^x$ is the configuration $\sigma$ with the spin at $x$ flipped.
The transition rates $c(x,\sigma)$ can be chosen arbitrarily subject to certain natural conditions (e.g., detailed balance), yet our attention in this work will be focused on the two most notable examples of Glauber dynamics:
\begin{enumerate}[(i)]
\item \emph{Metropolis}: $  c(x,\sigma) = \exp\Big(2\beta\sigma(x)\sum_{y \sim x}\sigma(y)\Big)  \;\wedge\; 1\; $.
\item \emph{Heat-bath}:  $\; c(x,\sigma) = \bigg[1+ \exp\Big(-2\beta\sigma(x)\sum_{y \sim x}\sigma(y)\Big)\bigg]^{-1}\;$.
\end{enumerate}
Each of these two flavors of Glauber dynamics has an intuitive and useful equivalent graphical interpretation: Assign i.i.d.\ rate-one Poisson clocks to the sites, and upon a clock ringing at some site $x$ do as follows:
\begin{enumerate}[(i)]
  \item \emph{Metropolis}: flip $\sigma(x)$ if the new state, $\sigma^x$, has a lower energy (that is, $\mu(\sigma^x)\geq \mu(\sigma)$), otherwise perform the flip with probability $\mu(\sigma^x)/\mu(\sigma)$.
  \item \emph{Heat-bath}: erase $\sigma(x)$ and replace it with a sample from the conditional distribution given the spins at its neighboring sites.
\end{enumerate}
It is easy to verify that the above chains are indeed ergodic and reversible with respect to the Gibbs distribution $\mu$.

The mixing time of Glauber dynamics has been extensively studied (see \S\ref{subsec:prelim-gap-sob} discussing the related relaxation time $\gap^{-1}$ and the references therein), yet as for cutoff, until recently the only setting where it was shown was heat-bath Glauber dynamics for Ising on the \emph{complete graph}~\cites{LLP,DLP}. There, the mean-field geometry reduced the problem to the analysis of a \oned birth-and-death chain (the sum-of-spins) which governs
the mixing of the entire system. While this further motivated the conjecture of Peres on cutoff for the Ising model on lattices (see~\cites{LLP,LPW}), it failed to provide insight for the dynamics on $(\Z/n\Z)^d$ even in the well-understood case of $d=1$.

In the companion paper~\cite{LS1} the authors verified the conjectured cutoff for Glauber dynamics for the Ising model on the $d$-dimensional torus $(\Z/n\Z)^d$ for any $d\geq 1$ and throughout the strong spatial mixing regime, which for $\Z^2$ extends all the way to the critical inverse-temperature $\beta_c = \frac12\log(1+\sqrt{2})$.

\subsection{Spectral gap and the logarithmic-Sobolev constant}\label{subsec:prelim-gap-sob}
The following quantities, defined next directly for Glauber dynamics for simplicity,
 provide useful analytic methods for bounding the mixing time of a chain.
The spectral gap and log-Sobolev constant of the
continuous-time Glauber dynamics are given by the following Dirichlet form
(see, e.g., \cites{Martinelli97,SaloffCoste}):
\begin{align}\label{eq-dirichlet-form}
\gap = \inf_f \frac{\mathscr{E}(f)}
{\var(f)}~,~\quad \sob = \inf_{f} \frac{\mathscr{E}(f)}{\operatorname{Ent}(f)}~,
\end{align}
where the infimum is over all nonconstant $f\in L^2(\mu)$ and
\begin{align*}
\mathscr{E}(f) &= \left<\mathscr{L} f,f\right>_{L^2(\mu)} = \frac12\sum_{\sigma,x} \mu(\sigma)c(x,\sigma)\left[f(\sigma^x)-f(\sigma)\right]^2\,,\\
\operatorname{Ent}(f) &= \E\left[f^2(\sigma) \log \left( f^2(\sigma) /\E f^2(\sigma)\right)\right] ~.
\end{align*}
It is well known (see e.g.\ \cites{DS,AF}) that for any finite ergodic reversible Markov chain $0<2\sob<\gap$ and  $\gap^{-1}\leq \tmix(1/\mathrm{2e})$.
In our case, since the sites are updated via rate-one independent Poisson clocks, we also have $\gap\leq 1$.

By bounding the log-Sobolev constant one may obtain remarkably sharp upper bounds not only for
 the total-variation mixing-time but also for the $L^2$-mixing (cf., e.g., \cites{DS1,DS2,DS,DS3,SaloffCoste}). The following theorem
of Diaconis and Saloff-Coste~\cite{DS}*{Theorem 3.7} (its form as follows appears
in~\cite{SaloffCoste}*{Theorem 2.2.5}, also see ~\cite{AF}*{Chapter 8}) demonstrates this powerful
method.
\begin{theorem}\label{thm-l2-sobolev}
Let $(Y_t)$ be a finite reversible continuous-time Markov chain with stationary distribution $\pi$. Let $s > 0$ and for some initial state $y_0$ set
\[ t = (4\sob)^{-1}\big|\log\log\left(1/\pi(y_0)\right)\big|^+ + \gap^{-1} s\,,\]
where $|a|^{+} = a \vee 0$. Then
\[ \left\|\P_{y_0} (Y_t\in \cdot)-\pi\right\|_{L^2(\pi)}\leq \exp(1-s)\,.\]
\end{theorem}

Starting from the late 1970's, a series of seminal papers by Aizenman, Dobrushin, Holley, Shlosman, Stroock et al.\
(see, e.g., \cites{AH,DobShl,Holley,HoSt1,HoSt2,Liggett,LY,MO,MO2,MOS,SZ1,SZ2,SZ3,Zee1,Zee2})
has developed the theory of the convergence rate of Glauber dynamics to stationarity
at the high temperature regime.
For more details on these works the reader is referred to the companion paper~\cite{LS1} as well as to~\cites{LScritical,Martinelli97,Martinelli04}, while in what follows we describe the result of Martinelli and Olivieri~\cites{MO,MO2} on the log-Sobolev constant, which is essential to our proofs.

\subsection{Strong spatial mixing and logarithmic-Sobolev inequalities}\label{sec:prelim-ssm}

Bounds on the log-Sobolev constant of the Glauber dynamics for the Ising model (as well as other spin systems) were
proved under a variety of increasingly general spatial mixing conditions.
We will work under the assumption of strong spatial mixing (or regular complete analyticity) introduced by Martinelli and Olivieri~\cite{MO}
as it holds for the largest known range of $\beta$.
\begin{definition}
For a set $\Lambda\subset \Z^d$ and constants $C,c>0$ we say that the property $\mathtt{SM}(\Lambda,C,c)$ holds if for any $\Delta\subset \Lambda$,
\[
\sup_{\tau,y} \ \Big\| \mu_\Lambda^\tau|_\Delta - \mu_\Lambda^{\tau^y}|_\Delta  \Big\|_{\tv} \leq C \mathrm{e}^{-c
\dist(y,\Delta)}\,,
\]
where the supremum is over all $y \in \partial \Lambda$ and $\tau \in \{\pm1\}^{\partial\Lambda}$ and where $\mu_\Lambda^\tau|_\Delta$ is the projection of the measure $\mu_\Lambda^\tau$ onto $\Delta$.
We say that \emph{strong spatial mixing} holds
for the Ising model with inverse temperature $\beta$ and external field $h$ on $\Z^d$ if there exist $C,c>0$ such that $\mathtt{SM}(\Lambda,C,c)$ holds for any \emph{cube} $\Lambda$.
\end{definition}
The above definition implies uniqueness of the Gibbs measure on the infinite lattice.
Moreover, strong spatial mixing holds for all temperatures when $d=1$ and for $d=2$ it holds whenever $h\neq 0$ or $\beta<\beta_c$.
As discussed in the introduction, this condition further implies a uniform lower bound on the log-Sobolev constant
of the Glauber dynamics on cubes under any boundary condition $\tau$ (see \cites{MO,MO2,Martinelli97}).
The proof of Theorem~\ref{mainthm-plusBC2d} will make use of the next generalization of this result to periodic boundary conditions,
i.e.\ the dynamics on the torus, obtained by following the original arguments as given in \cite{Martinelli97} with minor alterations
(see also \cites{Cesi,GZ,LY,Martinelli04}).

\begin{theorem}\label{thm-log-sobolev-torus}
Suppose that the inverse-temperature $\beta$ and external field $h$ are such that
the Ising model on $\Z^d$ has strong spatial mixing. Then there exists a constant
$\sobinf=\sobinf(\beta,h) > 0$ such that the Glauber dynamics for the Ising model on
$(\Z/n\Z)^d$ with any combination of arbitrary and periodic boundary conditions has a log-Sobolev constant  at least $\sobinf$ independent of $n$.
\end{theorem}

\section{Mixing in $L^1$ and $L^2$}\label{sec-l1-l2}

\subsection{Proof of Proposition~\ref{prop-l1-l2}}
To simplify the notation throughout the proof we write $\nu = \P(X_t \in \cdot)$ and $\nu_i = \P(X_t^i \in \cdot)$ for $i=1,\ldots,n$.

We will first show that $\left\|\nu - \pi\right \|_\tv \leq \sqrt{\ltwo_t}$. Note that by Cauchy-Schwarz,
\[ \| \nu - \pi \|_\tv = \frac12 \| \nu-\pi\|_{L^1(\pi)} \leq
\frac12 \| \nu-\pi\|_{L^2(\pi)}\,,\]
and at the same time, letting $\E_\pi$ denote expectation w.r.t.\ $\pi$, we have
\[ \| \nu-\pi\|_{L^2(\pi)}^2 = \E_\pi \left|\frac{\nu}{\pi}-1\right|^2 = \E_\pi \left|\frac{\nu}{\pi}\right|^2 - 1\,.\]
Since $\nu$ and $\pi$ are both product-measures,
\[ \E_\pi \left|\frac{\nu}{\pi}\right|^2 = \prod_{i=1}^n \E_{\pi_i}\left|\frac{\nu_i}{\pi_i}\right|^2
= \prod_{i=1}^n \left(1+ \|\nu_i-\pi_i\|^2_{L^2(\pi_i)}\right) \leq \exp(\ltwo_t)\,,\]
and we may conclude that
\[\| \nu - \pi \|_\tv \leq \frac12 \sqrt{\exp(\ltwo_t)-1} \,.\]
It is now clear that $|\nu-\pi\|_\tv \to 0$ as $\ltwo_t\to 0$, though a more useful relation between these quantities is the following.
When $\ltwo_t\geq 1$ put $\|\nu-\pi\|_\tv \leq 1$ trivially by the definition
of total-variation distance, while elsewhere one has $\frac12\sqrt{\mathrm{e}^x-1} \leq \sqrt{x}$ for $x \in [0,1]$. Hence,
\begin{equation}
  \label{eq-nu-pi-sqrt(mt)}
  \| \nu - \pi \|_\tv \leq \sqrt{\ltwo_t}\,.
\end{equation}

Next, we assume that~\eqref{eq-product-assumption} holds towards deriving the full form of~\eqref{eq-product-limit}.
It is worthwhile mentioning that to this end we may assume that $\ltwo_t \geq c_\delta$ for some $c_\delta >0$,
as otherwise one can infer~\eqref{eq-product-limit} by selecting a suitably small $c_\delta$ such that $2\Phi(\sqrt{\ltwo_t}/2)-1 < \delta/2$
while $\|\nu-\pi\|_\tv < \delta/2$ via~\eqref{eq-nu-pi-sqrt(mt)}.

Let $U_1,\ldots,U_n$ denote independent random variables drawn from $\pi_1,\ldots,\pi_n$ respectively and define
\begin{equation}
Y_{i}= Y_i(t) = \nu_i(U_i) / \pi_i(U_i)\,.
\end{equation}
Clearly by definition $\E Y_i = \sum_{x} \nu_i(x) = 1$ whereas
\begin{align*}
\var(Y_i) &= \sum_{x} \bigg | \frac{ \nu_i(x) }{ \pi_i(x)}-1 \bigg|^2 \pi_i(x) = \left\|\nu_i-\pi_i \right\|^2_{L^2(\pi_i)}.
\end{align*}
Moreover, the hypothesis~\eqref{eq-product-assumption} implies that $\|Y_i-1\|_\infty < \epsilon$ for all $1\leq i \leq n$. In particular,
\begin{align*}
\E |Y_i - 1|^3 \leq  \|Y_i - 1\|_\infty  \var(Y_i)  < \epsilon \var(Y_i)\,.
\end{align*}
Furthermore, noting that $Y_i>0$ with probability $1$ define $Z_i = \log Y_i$. By considering the Taylor series expansion of $Z_i$ we obtain that
\begin{align*}
  \E Z_i &= \E (Y_i - 1) - \tfrac12 \E (Y_i-1)^2 + O\left( \E |Y_i-1|^3  \right) = -\frac{1+O(\epsilon)}2 \var(Y_i)
\end{align*}
and similarly
\begin{align*}
  \E Z_i^2 &= \E  (Y_i - 1)^2 + O\left( \E |Y_i-1|^3  \right) = (1+O(\epsilon)) \var(Y_i)\,,
\end{align*}
where here and through the remainder of this proof the implicit constant in the $O(\cdot)$-notation is absolute.
The random variables $Z_i$ are independent and $\|Z_i\|_\infty = O(\epsilon)$ by~\eqref{eq-product-assumption}, hence the Berry-Esseen Theorem implies that
\[\sup_x \left| \P\left(\frac{\sum_{i=1}^n (Z_i -\E Z_i) }
{\sqrt {\sum_{i=1}^n \var(Z_i)}} < x\right) - \Phi(x)\right| < \frac{O(\epsilon)}{\sqrt{\ltwo_t}}\,.\]
Recalling definition~\eqref{eq-product-m-def} and that
$\var(Y_i) < \epsilon^2$ due to~\eqref{eq-product-assumption},
we have that
\begin{align}
\begin{array}{c}\mu\deq\sum_{i=1}^n \E Z_i = -(\tfrac12+O(\epsilon))\ltwo_t\,,\\
\noalign{\medskip}
\sigma^2\deq\sum_{i=1}^n \var(Z_i) = (1+O(\epsilon))\ltwo_t\,,
\end{array}\label{eq-ltwo-mu-sigma}
\end{align}
and altogether we find that for a suitably small $\epsilon=\epsilon(\delta)>0$ the distribution of $\sum_{i=1}^n Z_i$ becomes arbitrarily close to that of a Gaussian $\cN(\mu,\sigma^2)$.

To relate this variable to the $L^1$-distance of $\nu$ from $\pi$ observe that
\begin{align*}
  \left\|\nu - \pi\right \|_\tv
  &= \sum_{x_1,\ldots,x_n}
  \left| \nu(x_1,\ldots,x_n)-\pi(x_1,\ldots,x_n)\right|^- \\
& =  \sum_{x_1,\ldots,x_n}
\bigg| \prod_{i=1}^{n} \frac{\nu_i(x_i)}{\pi_i(x_i)}-1\bigg|^- \prod_{i=1}^{n} \pi_i(x_i) \\
& = \E \bigg| \prod_{i=1}^{n} Y_i -1\bigg|^-
 = \E \bigg|  \exp\bigg(\sum_{i=1}^{n} Z_i \bigg) -1\bigg|^{-}\,,
\end{align*}
where $|a|^{-}$ denotes $\max\{-a,0\}$.  Using the above argument we can now express the total-variation distance via a log-normal random variable, namely for every $\delta>0$ one can choose $\epsilon>0$ small enough such that
\begin{align}\label{eq-log-normal}
\left|\, \left\|\nu - \pi\right \|_\tv - \E\left|W-1\right|^-\,\right| < \delta/2\,,\quad\mbox{where}\quad \log W\sim \cN(\mu,\sigma^2)\,.
 \end{align}
Revisiting~\eqref{eq-ltwo-mu-sigma} we see that $\cN(\mu,\sigma^2)$ concentrates around $(-\frac12+O(\epsilon))\ltwo_t$. Hence, there is some large enough $B_\delta>0$ such that
if $\epsilon<1/B_\delta$ and $\ltwo_t> B_\delta$ then $\E|W-1|^- > 1- \frac{\delta}2$ and at the same time $2\Phi(\sqrt{\ltwo_t}/2)-1 > 1 - \frac\delta2$.

It is thus left to deal with the case where $\ltwo_t < B_\delta$. By choosing $\epsilon$ small enough we can
now let $\mu$ and $\sigma^2$ tend arbitrarily close to $-\frac12\ltwo_t$ and $\ltwo_t$, respectively, and consequently
rewrite~\eqref{eq-log-normal} with $\log W\sim \cN(-\frac12\ltwo_t,\ltwo_t)$.
The proof is then concluded by noting that
\begin{align*}
\E\left|W-1\right|^- &= \frac1{\sqrt{2\pi \ltwo_t}} \int_{-\infty}^0 (1-\mathrm{e}^x) \mathrm{e}^{-\frac{(x+\frac12\ltwo_t)^2}{2\,\ltwo_t}}dx
 \\
&= \P\left(\cN(-\tfrac12\ltwo_t,\ltwo_t) \leq 0\right) - \P\left(\cN(\tfrac12\ltwo_t,\ltwo_t) \leq 0\right) \\
&= 2\Phi(\sqrt{\ltwo_t}/2)-1\,,
 \end{align*}
 as required.
\qed

Using a standard Riesz-Thorin interpolation bound for the $L^\infty$-distance from equilibrium of a reversible Markov chain at time $t>0$ via the analogous $L^2$-distance at time $t/2$ we can now write down the following corollary specializing the above proposition to the case of a product of i.i.d.\ chains.
\begin{corollary}\label{cor-products}
Let $X_t=X_t^{(n)}$ denote a sequence of product chains
with stationary measures $\pi=\pi^{(n)}$, where the $n$-th chain in the sequence is made of $n$ i.i.d.\ copies of some finite ergodic Markov chain $Y_t=Y_t^{(n)}$.
Let $\phi$ be the stationary measure of $Y_t$, put $\phi_{\min} = \min_x \phi(x)$ and
let $\sob$ and $\gap$ denote the log-Sobolev constant and spectral gap of $Y_t$, respectively. If
\[ \log \phi_{\min}^{-1} \leq n^{o(\sob / \gap)}\]
 then $X_t$ has cutoff at $\frac1{2\, \gap}\log n$ with window $O(\gap^{-1} + \sob^{-1} |\log \log \phi_{\min}^{-1}|^+)$.
\end{corollary}
\begin{proof}
Set $t^\star = \frac1{2\,\gap} \log n$ and define the following time interval for $\gamma > 0$:
\begin{align*}
t^-=t^-(\gamma) &= t^\star - \gamma\, \gap^{-1}\,,\\
t^+=t^+(\gamma) &= t^\star   + \gamma\, \gap^{-1} + (4\sob)^{-1} |\log \log \phi_{\min}^{-1}|^+\,,
\end{align*}
while noting that for any fixed $\gamma>0$ we have $t^- = (1-o(1))t^\star$ and in addition $t^+ = (1+o(1))t^\star$ thanks to the hypothesis $\log \phi_{\min}^{-1} = n^{o(\sob/\gap)}$.

Following the notation of Proposition~\ref{prop-l1-l2}, each $X_t^i$ is an independent copy of $Y_t$, and we start by verifying the condition~\eqref{eq-product-assumption}. This will follow from the well-known fact (see, e.g.,~\cite{SaloffCoste}*{Eq.~(2.4.7)}) that if $Y_t$ is a reversible Markov chain with stationary distribution $\phi$
then for any $t>0$ one has
\[ \|\P(Y_t \in \cdot)-\phi\|_{L^\infty(\phi)} \leq
\|\P(Y_{t/2} \in \cdot)-\phi\|^2_{L^2(\phi)}\,.\]
Combining the hypothesis $\log \phi_{\min}^{-1} = n^{o(\sob/\gap)}$ with Theorem~\ref{thm-l2-sobolev}, it now follows
that for any $t \asymp \gap^{-1} \log n$ we have
\[  \max_{y_0}\|\P(Y_t \in \cdot)-\phi\|_{L^2(\phi)} \leq \exp\big(1-(1-o(1))t\,\gap\big)\,,\]
We conclude from the last two inequalities that for any $t\geq t^{-}$
\[ \max_{y_0}\|\P(Y_t \in \cdot)-\phi\|_{L^\infty(\phi)} \leq n^{-1/2+o(1)} = o(1)\,,\]
thus satisfying the prerequisite~\eqref{eq-product-assumption} for any $\epsilon>0$. Consequently, by~\eqref{eq-product-limit}
\[ |\P(X_t\in\cdot)-\pi|_\tv = 2\Phi(\sqrt{\ltwo_t}/2)-1 + o(1) \] and it is left to evaluate $ \ltwo_t = n \max_{y_0}\left\|\P_{y_0}(Y_t\in\cdot)-\phi\right\|^2_{L^2(\phi)}$
at $t^-$ and $t^+$.

For $t^-$, the standard lower bound on $L^2$-distance via the spectral gap (cf., e.g.,~\cite{LPW}) implies that
$ \max_{y_0} \left \| \P(Y_{t} \in\cdot)- \phi\right \|_{L^2(\phi)} \geq \exp(-t \, \gap) $, thus
$ \ltwo_{t^-} \geq \exp(\gamma)$, and so $|\P(X_{t^-}\in\cdot)-\pi|_\tv \to 1$ as $\gamma\to\infty$.

For $t^+$, another application of the log-Sobolev inequality in Theorem~\ref{thm-l2-sobolev} implies that
$\ltwo_{t^+} \leq \exp(2-2\gamma)$, thus $|\P(X_{t^+}\in\cdot)-\pi|_\tv \to 0$ as $\gamma\to\infty$.

Altogether, cutoff occurs at $t^\star$ with window $O(t^+-t^-)$, as required.
\end{proof}

\begin{remark*}
Corollary~\ref{cor-products} gives insight to the behavior of the Ising model on the torus $(\Z/n\Z)^d$.
Consider the (much simplified) model where the torus is partitioned
into boxes of side-length, say, $m=\log n$, and each box evolves independently with its own periodic boundary conditions
(all long range interactions are completely absent from this model).
Observe the following:
\begin{compactitem}[\indent$\bullet$]
  \item The dynamics is a product of $N=(n/m)^d=n^{d-o(1)}$ i.i.d.\ copies of Glauber dynamics for the Ising model on $(\Z/m\Z)^d$, denoted by $(Y_t)$.
  \item Let $\phi$ denote the stationary measure of $(Y_t)$ and note that its state space contains $2^{m^d}$ configurations and so $\log \phi_{\min}^{-1} = O(m^d) = n^{o(1)}$.
  \item The log-Sobolev constant of $(Y_t)$ is known to be uniformly bounded away from $0$, hence in particular $\log \phi_{\min}^{-1} = N^{o(\sob/\gap)}$.
\end{compactitem}
Appealing to Corollary~\ref{cor-products} we conclude that there is cutoff at $(2\lambda_m)^{-1}\log n$ with window $O(\log\log n)$, where $\lambda_m$ is the gap of the dynamics on $(\Z/m\Z)^d$.
Given that $\lambda_m$ converges to the spectral-gap of the infinite-volume dynamics as $m\to\infty$ (as was indeed shown in~\cite{LS1}), this picture of cutoff coincides with the actual behavior of the Ising model on the torus proved in~\cite{LS1}.
\end{remark*}

Note that, as the Ising model does of course feature long range interactions (albeit weak ones at high temperatures), a significant part of the proof will entail controlling these interactions to enable an application of Proposition~\ref{prop-l1-l2}.

\subsection{Supports of random maps}
In what follows we place Lemma 3.8 of~\cite{LS1} in a general context.
Recall that if $K$ is a transition kernel of a finite Markov chain then a \emph{random mapping representation} for $K$ is a pair $(g,W)$ where $g$ is a deterministic map and $W$ is a random variable such that $\P(g(x,W)=y)=K(x,y)$ for all $x,y$ in the state space of $K$. It is well-known (and easy to see) that such a representation always exists.

\begin{definition}[\emph{Support of a random mapping representation}]\label{def-mc-support} Let $K$ be a Markov chain on a state space $\Sigma^V$ for some finite sets $\Sigma$ and $V$. Let $(g,W)$ be a random mapping representation for $K$. The \emph{support} corresponding to $g$ for a given value of $W$ is the minimum subset $\Lambda_{W} \subset V$ such that $g(\cdot,W)$ is determined by $x(\Lambda_{W})$ for any $x$, i.e.,
\[ g(x,W) = f_{W}(x(\Lambda_{W})) \mbox{ for some $f_{W}:\Sigma^{\Lambda_{W}}\to \Sigma^{V}$ and all $x$.}\]
That is, $v \in \Lambda_{W}$ if and only if there exist some $x,x'\in \Sigma^V$ differing only at coordinate $v$ such that $g(x,W)\neq g(x',W)$.
\end{definition}

\begin{lemma}\label{lem-support} Let $K$ be a finite Markov chain and let $(g,W)$ be a random mapping representation for it. Denote by $\Lambda_{W}$ the support of $W$ w.r.t.\ $g$ as per Definition~\ref{def-mc-support}. Then for any distributions $\phi,\psi$ on the state space of $K$,
\begin{align*}
\left\| \phi K -\psi K\right\|_\tv &\leq \int \left\|\phi|_{\Lambda_{W}}-\psi|_{\Lambda_{W}} \right\|_\tv d \P({W}).
 \end{align*}
\end{lemma}
\begin{proof}
Let $\Omega$ denote the state space of $K$.
Following the definition of the support set $\Lambda_W$, let $f_{W}$ denote the deterministic map which satisfies $g(x,W) = f_{W}(x(\Lambda_{W}))$ for all $x \in \Omega$.

By definition of total-variation distance,
 \begin{align*}
\left\|\phi K - \psi K\right\|_\tv &=
  \max_{\Gamma\subset \Omega} \left[(\phi K)(\Gamma) -
(\psi K)(\Gamma)\right]\,.
\end{align*}
Let $X,Y$ be random variables distributed according to $\phi,\psi$ respectively. The above is then equal to
\begin{align*}
  \max_{\Gamma\subset \Omega} &\int\left[ \P\big(g(X,W)\in\Gamma\big) -\P\big(g(Y,W)\in \Gamma\big)\right]\,d\,\P({W}) \\
\leq &\int  \max_{\Gamma\subset \Omega} \left[ \P\big(f_{W}(X(\Lambda_{W}))\in\Gamma\big) -\P\big(f_{W}(Y(\Lambda_{W}))\in \Gamma\big)\right]\,d\,\P({W})\\
\leq &\int \left\| \P\big(X(\Lambda_{W})\in\cdot\big) - \P\big(Y(\Lambda_{W})\in\cdot\big)\right\|_\tv\,d\,\P({W})\\
= &\int \left\| \phi|_{\Lambda_{W}} - \psi|_{\Lambda_{W}}\right\|_\tv\,d\,\P({W})\,,
 \end{align*}
where the second inequality used the fact that when taking a projection of two measures their total-variation can only decrease.
\end{proof}


\section{Cutoff for on arbitrary graphs via log-Sobolev inequalities}\label{sec:arbitrary}

\subsection{Sparse supports}
The main step in proving Theorem~\ref{mainthm-arbitrary} would be to break up the dependencies between the spins and express the total-variation distance of the dynamics from equilibrium in terms of its projection on a collection of well-separated small clusters of spins, which we refer to as a \emph{sparse} set of spins. Thereafter the proof would proceed by showing that these clusters are essentially independent, justifying an application of Proposition~\ref{prop-l1-l2} to reduce $L^1$-mixing to $L^2$-mixing for this projection. We first define the notion of a sparse set. Throughout the proof, let $V$ denote the vertex set of $G$ and let $E$ denote its set of edges.
\begin{definition}[\emph{Sparse set}]\label{def-sparse-set}
We say the set $\Lambda \subset V$ is \emph{sparse} if it can be partitioned into (not necessarily connected) components $\{A_i\}$ so that
\begin{enumerate}
  [\indent\!\!\!1.]
  \item Each $A_i$ contains at most $\rho^3 \log n$ vertices.
  \item The diameter of every $A_i$ in $G$ is at most $\frac12\log^2 n$.
  \item The distance in $G$ between any distinct $A_i,A_j$ is at least $25 \Delta \sobinf^{-1}\log n $.
\end{enumerate}
Let $\sparse=\sparse(G)=\{ \Lambda\subset V: \mbox{$\Lambda$ is sparse}\}$.
\end{definition}

We can now state the main result of this subsection, which provides an upper bound on the $L^1$-distance of the dynamics
from equilibrium in terms of the corresponding quantity for its projection onto a \emph{sparse} set of sites.
Recalling the notation of Theorem~\ref{mainthm-arbitrary}, we let $\Delta$ denotes the maximal degree,
$\sobinf$ is the minimal log-Sobolev constant over
all induced subgraphs $H\subset G$ whose diameter is at most $\log^2 n$ (under free boundary conditions), and $\rho$ is the
maximal size of a ball of radius $10\Delta \sobinf^{-1}\log n$ around a vertex.

\begin{theorem}\label{thm-xt-bound-a-sparse}
Let $(X_t)$ be the Glauber dynamics on $G$ and let $\mu$ be its stationary measure. Let $7 \Delta \sobinf^{-2} \log\rho \leq s \leq 2 \sobinf^{-1} \log n$ and $t>0$. Then there exists some distribution $\nu$ on $\sparse$ such that for large enough $n$
\begin{align*}
\left\| \P_{\sigma_0}(X_{t+s}\in\cdot)-\mu\right\|_\tv &\leq \int_{\sparse}\left\|\P_{\sigma_0}(X_t(\Lambda)\in\cdot)-\mu|_{\Lambda} \right\|_\tv\, d\nu(\Lambda) +3 n^{-10}\,,
 \end{align*}
 and moreover $\nu(\{\Lambda : v \in \Lambda\}) \leq \rho^{-10}$ for every vertex $v\in V$.
\end{theorem}
The remainder of this subsection will be devoted to the proof of the above theorem, the first ingredient of which would be to consider a close variant of the dynamics which effectively turns distant sites into independent: We refer to this as the \emph{Barrier dynamics}, as it is achieved roughly by fixing the boundary on the perimeter of a ball centered at each vertex (forming a barrier on the propagation of information).

Recall that $B_v(\ell) = \{ w \in V : \dist(v,w)\leq \ell\}$ denotes the ball of radius $\ell$ about the vertex $v$, and similarly let $\partial B_v(\ell)$ denote the boundary of this ball, i.e.\ the subset $\{w \in V : \dist(v,w)=\ell\}$. For a subset $A \subset V$ define $B_A(\ell)$ and $\partial B_A(\ell)$ analogously, e.g.\ $B_A(\ell) = \cup_{v\in A}B_v(\ell)$.

Throughout this section let
\begin{align}
  \label{eq-r-def}
  r = \lfloor 10 \Delta \sobinf^{-1} \log n \rfloor\,,
\end{align}
where $\sobinf$ was defined in Theorem~\ref{mainthm-arbitrary} to be the minimum of the log-Sobolev constants of the Glauber dynamics over all induced subgraphs of $G$ with diameter at most $\log^2 n$. Note that by definition $\rho = \max_{v\in V} |B_v(r)|$.
\begin{definition}
  [\emph{Barrier-dynamics}]\label{def-barrier}
Let $(X_t)$ be the Glauber dynamics for the Ising model on $G$.
Define the corresponding \emph{barrier-dynamics} as the following coupled Markov chain:
\begin{enumerate}
  \item For each vertex $u$, let $\Omega_u = \{\pm 1\}^{B_r(u)}$. The state set of the barrier dynamics is the cartesian product of $\{\Omega_u : u \in V\}$ and its marginal on $\Omega_u$ is the Ising model on $B_u(r)$ with free boundary (to be thought of as if a barrier disconnected $B_u(r)$ from the remainder of $G$).
  \item The coupling of the dynamics with the standard Glauber dynamics $(X_t)$ is defined as follows:
  The initial configuration of each $\Omega_u$ is obtained by projection the initial configuration of $X$ onto $B_u(r)$ for each $u$.
  Whenever the Glauber dynamics updates a site $v$ via a variable $I\sim U[0,1]$ in the standard dynamics (occurring according to an independent unit-rate Poisson clock), each $\Omega_u$ such that $v \in B_u(r)$ updates its copy of $v$ via the same update variable $I$.
\end{enumerate}
\end{definition}
For any $s>0$, the barrier dynamics gives rise to a randomized operator $\cG_s$ on $\{\pm1\}^V$ obtained by projecting its configuration onto the centers of the balls $B_u(r)$. That is, in order to compute $\cG_s(\sigma)$, run the barrier dynamics for time $s$ starting from the initial configuration formed from $\sigma$ in the obvious manner and denote its final configuration by $\{ \sigma_u \in \Omega_u : u\in V\}$.
The output of $\cG_s(\sigma)$ assigns each $u\in V$ its value in $\sigma_u$.

\begin{lemma}\label{lem-barrier-couple}
Let $t_0 = 2\, \sobinf^{-1}\log n$. For any sufficiently large $n$, the barrier-dynamics and original Glauber dynamics are coupled up to time $t_0$
except with probability $n^{-10}$. That is, except with probability $n^{-10}$, for any $X_0$ we have $X_{s} = \mathcal{G}_s(X_0)$ simultaneously for all $s \leq t_0$.
\end{lemma}
\begin{proof}
Let $(X_t)$ denote the Glauber dynamics on $G$ and let $(\tilde{X}_t)$ denote the barrier-dynamics coupled to $(X_t)$ as defined above.
Clearly, for any $l \leq r$ and $v\in V$ we have that $X_s \equiv \tilde{X}_s$ on $B_v(l)$ as long as none of the sites comprising $\partial B_v(l)$ were updated by time $s$. Therefore, if we assume that $X_s(v)\neq \tilde{X}_s(v)$ for some $v\in V$ and $s \leq t_0$ then there necessarily must exist a path of adjacent sites $u_1,\ldots,u_\ell$ connecting $v$ to $\partial B_v(r)$ and a sequence of times $t_1<\ldots<t_\ell\leq s$ such that site $u_i$ was updated at time $t_i$ (and as such $\ell \geq r$). Summing over all $\Delta^\ell$ possible paths originating from $v$, and accounting for the probability that the $\ell$ corresponding unit-rate Poisson clocks fire sequentially before time $s\leq t_0$ it then follows that
\begin{align*}
\P\bigg(\bigcup_{v\in V} \left\{ X_t(v) \neq \tilde{X}_t(v)\right\}\bigg) &\leq n \sum_{\ell \geq r} \Delta^\ell \, \P(\Po(t_0) \geq \ell)
 \leq 2 n \mathrm{e}^{-t_0} \frac{(\Delta t_0)^r }{r!} \,,
\end{align*}
where the last inequality followed from the fact that $r \geq 2 \Delta t_0$. Moreover, since
$r = (5+o(1)) \Delta t_0$ we further have $(\Delta t_0)^r/r! = (\mathrm{e}/5)^{(1+o(1))r}$. Using the facts $\mathrm{e}/5 \leq \mathrm{e}^{-3/5}$ and $r\geq (20+o(1))\log n$ (recall that $\Delta \geq 2$ and $\sobinf \leq 1$) we can now infer the required result (with room to spare) for large $n$.
\end{proof}

We next define the notion of an \emph{update support}, specializing the notion of the support of the random mapping representation from Definition~\ref{def-mc-support} to the barrier-dynamics map.
\begin{definition}[\emph{Update support}]\label{def-support} Let $W_s$ be an update sequence for the barrier-dynamics between times $(0,s)$.
The \emph{support} of $W_s$ is the minimum subset $\Lambda_{W_s} \subset V$ such that $\mathcal{G}_s(x)$ is a function of $x(\Lambda_{W_s})$ for any $x$, i.e.,
\[ g_{W_s}(x) = f_{W_s}(x(\Lambda_{W_s})) \mbox{ for some $f_{W_s}:\{\pm1\}^{\Lambda_{W_s}}\to \{\pm1\}^{V}$ and all $x$.}\]
For any given $W_s$ its support $\Lambda_{W_s}$ is uniquely defined.
\end{definition}
The following lemma demonstrates the merit of considering the update support.
\begin{lemma}\label{lem-xt-bound-aw} Let $(X_t)$ be the Glauber dynamics on $G$, set $t_0 = 2\, \sobinf^{-1}\log n$ and let $W_s$ be the random update sequence for the barrier-dynamics
along an interval $(0,s)$ for some $s \leq t_0$. For large enough $n$ and any $\sigma_0$ and $t>0$,
\begin{align*}
\left\| \P_{\sigma_0}(X_{t+s}\in\cdot)-\mu\right\|_\tv &\leq \int\!\!\left\|\P_{\sigma_0}(X_t(\Lambda_{W_s})\in\cdot)-\mu|_{\Lambda_{W_s}} \right\|_\tv d \P({W_s})+ 2 n^{-10}.
 \end{align*}
\end{lemma}
\begin{proof}
Let $(X_t)$ be Glauber dynamics at time $t$ started from $X_0=\sigma_0$ and let $Y \in \Omega$ be distributed according to $\mu$.
By considering the random map $\mathcal{G}_s$ as a discrete Markov chain on $\Omega$,
an application of Lemma~\ref{lem-support} gives that
\begin{align*} \left\|\P(\mathcal{G}_s(X_t)\in\cdot) - \P(\mathcal{G}_s(Y)\in\cdot)\right\|_\tv
\leq \!\int\!\! \left\| \P\big(X_t(\Lambda_{W_s})\in\cdot\big) - \mu|_{\Lambda_{W_s}}\right\|_\tv d\P({W_s})\,.
 \end{align*}
Since $s \leq t_0$, two applications of Lemma~\ref{lem-barrier-couple} will now couple $X_{t+s}$ with $\mathcal{G}_s(X_t)$
and similarly $\mathcal{G}_s(Y)$ with Glauber dynamics run from $Y$ for time $s$ (having the stationary distribution $\mu$)
except with probability $n^{-10}$:
\begin{align*}
\left\|\P(X_{t+s}\in\cdot) - \P(\mathcal{G}_s(X_t)\in\cdot)\right\|_\tv &\leq n^{-10}\,,\\
\left\|\P(\mathcal{G}_s(Y)\in\cdot)-\mu\right\|_\tv &\leq n^{-10}\, .
\end{align*}
Combining these estimates, it follows that
 \begin{align*}
\left\|\P(X_{t+s}\in\cdot) - \mu\right\|_\tv &\leq\left\|\P(\mathcal{G}_s(X_t)\in\cdot) - \P(\mathcal{G}_s(Y)\in\cdot)\right\|_\tv + 2n^{-10} \\
&\leq\int \left\| \P\big(X_t(\Lambda_{W_s})\in\cdot\big) - \mu|_{\Lambda_{W_s}}\right\|_\tv\,d\,\P({W_s})+ 2 n^{-10}\,,
\end{align*}
as required.
\end{proof}

\begin{lemma}\label{lem-sparse-prob}
Let $\mathcal{G}_s$ be the barrier-dynamics operator, let $W_s$ be its update sequence up to time $s$ for some $s \geq 7 \Delta \sobinf^{-2}\log \rho$, and let $\sparse$ be the collection of sparse sets of $G$.
For sufficiently large $n$ we have $\P(\Lambda_{W_s} \in \sparse) \geq 1- n^{-10}$ and furthermore $\P(v\in\Lambda_{W_s}) \leq \rho^{-10}$ for every $v\in V$.
\end{lemma}
\begin{proof}
Let $U_v$ for $v\in V$ denote the event that the outputs of the barrier-dynamics operator $\mathcal{G}_s$ computed on the all-plus and all-minus inputs disagree on the value of $v$. Observe that if $\cG_s$ does agree on $v$ under the two extreme inputs of all-plus and all-minus then by monotonicity the initial configuration of $B_v(r)$ has no influence over this spin, and in particular,
\begin{equation}
  \label{eq-support-B-cover}
  \Lambda_{W_s} \subset \bigcup\left\{ B_v(r) : v \in V\mbox{ such that $U_v$ holds}\right\} \deq \overline{\Lambda}_{W_s}\,.
\end{equation}
Since the family of sparse sets monotone-decreasing, it clearly suffices to establish both statements of the lemma for the super-set $\overline{\Lambda}_{W_s}$ governed by the indicators $\{\one_{U_v} : v\in V\}$.

To estimate the probability of the event $U_v$ let $(\tX_t^+)$ and $(\tX_t^-)$ be two instances of the barrier-dynamics projected to $B_v(r)$ and starting from the all-plus and all-minus states respectively, coupled to each another via the monotone coupling.
These projections are simply the standard Glauber dynamics on $B_v(r)$, hence letting $\pi = \mu_{B_v(r)}^\emptyset$ we get
\begin{align*}
\P(U_v) &= \P\big(\tX^+_s(v) \neq \tX^-_s(v)\big) \leq
\| \P(\tX^+_s\in\cdot) - \pi \|_\tv + \| \P(\tX^-_s\in\cdot) - \pi \|_\tv \\
&\leq \tfrac12\|\P(\tX^+_s\in\cdot) - \pi\|_{L^2(\pi)} + \tfrac12\|\P(\tX^-_s\in\cdot) - \pi\|_{L^2(\pi)} \,.
\end{align*}
Due to Theorem~\ref{thm-l2-sobolev}, 
if the all-plus state $\underline{1}$ has stationary measure at most $\mathrm{e}^{-1}$ (clearly the case for large $n$ since $|B_v(r)|\geq r \to \infty$ and the interactions are finite) then for any $s >0$
\begin{equation}
   \label{eq-log-sob-support-application}
   \left\|\P (\tX^+_s\in \cdot)-\pi\right\|_{L^2(\pi)}\leq \exp\bigg(1-\gap\left(s-\frac1{4 \sob} \log\log \frac{1}{\pi(\underline{1})}\right)\bigg)\,,
 \end{equation}
where $\lambda$ and $\sob$ are the spectral gap and log-Sobolev constant resp.\ of the Glauber dynamics on $B_v(r)$. Recalling that trivially $\pi(\underline{1}) \geq 2^{-|B_v(r)|}$, by definition of $\rho$ we have
$\log\log(1/\pi(\underline{1})) \leq \log \rho$, and since $\gap \geq \sob \geq \sobinf$ the assumption on $s$ easily gives
\[ s-\frac1{4 \sob} \log\log \frac{1}{\pi(\underline{1})} \geq \left(7\Delta\sobinf^{-1} - \tfrac14\right) \gap^{-1} \log \rho\,,\]
hence for large $n$ we can absorb the pre-factor of $\mathrm{e}$ in the r.h.s.\ of~\eqref{eq-log-sob-support-application} and obtain (with room to spare) that
\[ \|\P(\tX^+_s\in\cdot) - \pi\|_{L^2(\pi)} \leq \rho^{-6 \Delta /  \sobinf} \,.\]
By the exact same argument we also have $\|\P(\tX^-_s\in\cdot) - \pi\|_{L^2(\pi)} \leq \rho^{-6\Delta/\sobinf}$ and it now follows that
\begin{align}
  \label{eq-Uv-bound}
  \P(U_v) \leq \rho^{-6\Delta / \sobinf}\,.
\end{align}
Going back to the definition of $\overline{\Lambda}_{W_s}$ in~\eqref{eq-support-B-cover} we can infer that for any $v\in V$,
\[ \P\left(v \in \overline{\Lambda}_{W_s}\right) \leq |B_v(r)|\rho^{-6\Delta/\sobinf} \leq \rho^{1-6\Delta/\sobinf} \leq \rho^{-10}\,,\]
where the last inequality (which used the fact that $\Delta\geq2$) implies the sought upper bound on $\P(v\in \Lambda_{W_s})$.

We now wish to show that $\P(\overline{\Lambda}_{W_s} \in \sparse) \geq 1-n^{-10}$. To see this, denote by $\Xi$ the event that there exists some sequence $S$ of
$ \ell = \lfloor \log n \rfloor$ points, $S = (v_1,v_2,\ldots,v_\ell)$, satisfying the following:
 \begin{enumerate}[1.]
 \item \label{it-S-dist} For all $i$ we have $2r < \dist_G\big(\{v_1,\ldots,v_i\},v_{i+1}\big) \leq 5r$.
 \item \label{it-S-support} For all $i$ the event $U_{v_i}$ holds.
 \end{enumerate}
Notice that for any sequence $S$ satisfying Item~\ref{it-S-dist},
given $v_1,\ldots,v_j$ there are at most $\sum_{i\leq j} |B_{v_i}(5r)| \leq j \rho^5$ choices for the vertex $v_{j+1}$. Furthermore, the balls $\{B_{v_i}(r)\}$ are pairwise disjoint and as a consequence the events $\{U_{v_i}\}$ are mutually independent. As $\P(U_{v_i}) \leq \rho^{-6\Delta/\sobinf} \leq \rho^{-12}$ for all $i$ we now get
\begin{align*}
 \P(\Xi) &\leq n (\ell-1)!\rho^{5(\ell-1)-12\ell}
 \leq n \left(\ell \rho^{-7}\right)^{\ell} \leq n \rho^{-6\ell}\\
 &\leq n^{1- (6-o(1))\log\rho} < n^{-\log\log n}
\end{align*}
for large enough $n$, as in that case one has $\rho \geq r  \geq \lceil\log n\rceil$.

Next, let $\Xi'$ denote the event that there exists some sequence $S'$ of
\[ \ell' = \bigg\lfloor 4\,\frac{\sobinf}{\Delta}\, \frac{\log n}{\log\rho}\bigg\rfloor \]
points, $S' = (v_1,v_2,\ldots,v_{\ell'})$, satisfying the following:
 \begin{enumerate}[1'.]
 \item \label{it-S'-dist} For all $i$ we have
 $\dist_G\big(\{v_1,\ldots,v_i\},v_{i+1}\big) > 2r$ and $\dist_G\big(v_i,v_{i+1}\big) \leq 5r$.
 \item \label{it-S'-support} For all $i$ the event $U_{v_i}$ holds.
 \end{enumerate}
We now argue as above: For any sequence $S$ satisfying Item~\ref{it-S'-dist}', given $v_j$ there are at most $|B_{v_j}(5r)| \leq \rho^5$ choices for the vertex $v_{j+1}$, and as before the balls $\{B_{v_i}(r)\}$ are pairwise disjoint and so the events $\{U_{v_i}\}$ are independent. Plugging in the fact that $\P(U_{v_i}) \leq \rho^{-6\Delta/\sobinf} $ for all $i$ now yields
\[ \P(\Xi') \leq n \rho^{5(\ell'-1)} \rho^{-6\Delta \sobinf^{-1}\ell'}
 \leq n \rho^{-3\Delta\sobinf^{-1}\ell'}
 = n \rho^{-(12-o(1))\frac{\log n}{\log\rho}} = n^{-11+o(1)}\,.\]

Altogether, for sufficiently large $n$ we have $\P(\Xi \cup \Xi') < n^{-10}$. Conditioned on the fact that neither $\Xi$ nor $\Xi'$ occurred, and
consider the partition of $\overline{\Lambda}_{W_s}$ to components $\{A_i\}$ obtained by repeatedly applying the rule whereby one identifies the components of $B_u(r)$ and $B_v(r)$ if $\dist_G(u,v) \leq 5r$.
Clearly, if $R = \{ v \in V : U_v \mbox{ holds}\}$ then each $A_i$ can contain at most $\rho^2 \ell$ vertices of $R$, otherwise one could greedily find in $A_i$ a sequence $S$ satisfying the conditions specified in the event $\Xi$. Since $\overline{\Lambda}_{W_s} = \bigcup_{v\in R} B_v(r)$ we get that
\[ |A_i| \leq \rho^3 \ell \leq \rho^3\log n\quad\mbox{ for all $i$}\,.\]
Similarly, for every $u,v\in R$ that belong to the same $A_i$ we must have $\dist_G(u,v) \leq 2r \ell'$ or else one could find in $A_i$ a sequence $S'$ satisfying the conditions specified in the event $\Xi'$. Thus, taking into account the inclusion of an extra $B_v(r)$ for all $v\in R$ when forming $\overline{\Lambda}_{W_s}$ we infer that for all $i$
\[ \diam_G(A_i) \leq 2r \ell' + 2r \leq \frac{80+o(1)}{\log \rho}\log^2 n < \frac12 \log^2 n \,,\]
where the last inequality holds for large enough $n$ as $\rho$ diverges with $n$.
Finally, if $u,v\in R$ belong to distinct components $A_i,A_j$ then by definition $\dist_G(u,v) > 5r$. Thus, accounting for the addition of $\{B_u(r) : u \in R\}$ when forming $\overline{\Lambda}_{W_s}$ gives that for all $i\neq j$ and sufficiently large $n$,
\[ \dist_G(A_i,A_j) > 3r > 25 \Delta \sobinf^{-1}\log n\,.\]
We conclude that $\P(\overline{\Lambda}_{W_s} \in \sparse) \geq 1-n^{-10}$, as required.
\end{proof}

The statement of Theorem~\ref{thm-xt-bound-a-sparse} is an immediate corollary of Lemmas~\ref{lem-xt-bound-aw} and~\ref{lem-sparse-prob}, obtained by setting $\nu(\Lambda) = \P(W_s : \Lambda_{W_s} = \Lambda)$ for $\Lambda\in \sparse$.

\subsection{Convergence in $L^1$ of the dynamics projected on sparse sets}
So far we showed an upper bound on the $L^1$-distance of the dynamics from equilibrium in terms
of its projection onto a sparse set of sites. The goal of this subsection is to provide the following estimate for the latter quantity:
\begin{theorem}\label{thm-tv-project}
Let $\Lambda\in \sparse$ and let $A_1,\ldots,A_L$ be its partition to components as per Definition~\ref{def-sparse-set}.
Let $(X^i_t)$ be the Glauber dynamics on $A_i^+ = B_{A_i}(r)$ with free boundary and the original interactions and external field, and let $\pi_i = \mu^\emptyset_{A_i^+}$ denote its stationary distribution. Further define
\begin{align*}
\lltwo_t(A_i,\sigma_0) &\deq \left\|\P_{\sigma_0}(X^i_{t}(A_i)\in\cdot)- \pi_i|_{A_i}  \right\|^2_{L^2(\pi_i|_{A_i})}\,,\\
\ltwo_t(\Lambda,\sigma_0) &\deq \sum_i \lltwo_t(A_i,\sigma_0)\,.
 \end{align*}
Then for any $\sigma_0$ and $\frac14 \log n \leq t \leq t_0$, the dynamics $(X_t)$ on $G$ satisfies
\[ \left\|\P_{\sigma_0}(X_t(\Lambda)\in\cdot)-\mu|_{\Lambda} \right\|_\tv = 2\Phi\big(\sqrt{\ltwo_t(\Lambda,\sigma_0)}/2\big)-1 + o(1)\,,\]
where the $o(1)$-term tends to $0$ as $n\to\infty$.
\end{theorem}

The following lemma will allow us to relate the projection of the dynamics onto a sparse set $\Lambda \in \sparse$ to the product-chain of the dynamics on the component partition corresponding to $\Lambda$.
\begin{lemma}\label{lem-sparse-product}
Let $\Lambda\in \sparse$ and let $A_1,\ldots,A_L$ be its partition to components as per Definition~\ref{def-sparse-set}. Set $A_i^+ = B_{A_i}(r)$,
let $(X_t^*)$ be the product chain of Glauber dynamics on the graphs induced on $A_1^+,\ldots,A_L^+$ independently (with free boundary and the original interactions and external field), and denote its stationary distribution by $\pi=\prod_i\mu_{A_i^+}^\emptyset$. Then for any $\sigma_0$ and $t \leq t_0$,
\[ \Big| \left\|\P_{\sigma_0}(X_t(\Lambda)\in\cdot)-\mu|_{\Lambda} \right\|_\tv -
 \left\|\P_{\sigma_0^*}(X^*_t(\Lambda) \in\cdot)-\pi|_\Lambda \right\|_\tv \Big| < n^{-2/3}\,,\]
 where $\sigma_0^* = \sigma_0(\cup_i A_i^+)$.
\end{lemma}
\begin{proof}[Proof of lemma]
Let $(X_t)$ be the original Glauber dynamics $(X_t)$ for the graph $G$ and consider its standard coupling
to the dynamics $(X_t^*)$, that is, run the two chains via the same unit-variables and Poisson clocks for the site updates. Crucially, the pairwise distances between the components $A_i$ are all strictly larger than $2 r$ and so the sets $A_i^+$ are pairwise disjoint, thus this coupling indeed retains the independence between the coordinates of the chain $X_t^*$. Moreover, the proof of Lemma~\ref{lem-barrier-couple} implies that
the projections onto $A_1,\ldots,A_L$ are maintained identical in this coupling until time $t_0$ except with probability at most $n^{-9}$ (the difference from the original estimate of $n^{-10}$ in that lemma is due to a union bound over the $L \leq n$ components). That is, with probability at least $1-n^{-9}$
\[ X_t(\Lambda) = X_t^*(\Lambda)\mbox{ for all $t\in[0,t_0]$}\,, \]
and in particular, letting $\sigma^*_0$ be the configuration derived from $\sigma_0$ in the obvious manner we get that for any $t \leq t_0$,
\begin{align}
  \label{eq-xt-xt*-distance}
 \| \P_{\sigma_0}(X_t(\Lambda)\in\cdot) - \P_{\sigma_0^*}(X_t^*(\Lambda)\in\cdot) \|_\tv \leq n^{-9}\,.
\end{align}
To relate $\mu|_\Lambda$ to $\pi|_\Lambda$ we first argue that $X_{t_0}$ is well-mixed under total-variation distance. Indeed, if $(X_t^+)$ and $(X_t^-)$ are instances of the dynamics starting from the all-plus and all-minus states resp.\ then for any $v\in V$,
\[
\P\left(X^+_{t_0}(v) \neq X^-_{t_0}(v)\right) \leq
\P\left(\tX^+_{t_0}(v) \neq \tX^-_{t_0}(v)\right) + n^{-10}\,,
\]
where $(\tX_t)$ denotes the Glauber dynamics on the graph $B_v(r)$ with interactions and external field inherited from $G$, and the additive term of $n^{-10}$ is justified by the same coupling argument given above. The event $\{\tX^+_{t_0}(v) \neq \tX^-_{t_0}(v)\}$ was considered in Lemma~\ref{lem-sparse-prob}, there denoted by $U_v$, and exactly the calculation that led Eq.~\eqref{eq-Uv-bound} valid for time $s \geq 5\Delta\sobinf^{-2}\log\rho$ now gives that
\[ \P\left(\tX^+_{t_0}(v) \neq \tX^-_{t_0}(v)\right) \leq n^{-7/4}\,.\]
Taking a union bound over the vertices it now follows that
\[   \P\left(X^+_{t_0} \neq X^-_{t_0}\right) \leq n^{-3/4} + n^{-10} = O(n^{-3/4})\,, \]
and in particular
\begin{equation}
  \label{eq-dynamics-mixing}
\max_{\sigma_0 } \| \P_{\sigma_0}(X_{t_0} \in \cdot) - \mu \|_\tv = O(n^{-3/4})\,.
\end{equation}
Let $X^*_0$ be a configuration on $\cup_i A_i^+$ distributed according to $\pi$ and extend it arbitrarily to a configuration $X_0$ on $G$.
As before, with probability at least $1-n^{-9}$ we can couple $(X_t),(X_t^*)$ so that they would agree on $\Lambda = \cup_i A_i$ throughout the time interval $[0,t_0]$. Since $X_{t_0}^*(\Lambda) \sim \pi|_\Lambda$ and total-variation distance can only decrease on a marginal we can infer from~\eqref{eq-dynamics-mixing} that
\begin{align*}
\| \pi|_\Lambda - \mu|_\Lambda \|_\tv &=
\| \P_{X_0^*}(X^*_{t_0}(\Lambda) \in \cdot) - \mu|_\Lambda \|_\tv \\ &\leq
\| \P_{X_0}(X_{t_0}(\Lambda) \in \cdot) - \mu|_\Lambda \|_\tv + n^{-9} \\
&\leq \| \P_{X_0}(X_{t_0} \in \cdot) - \mu \|_\tv + n^{-9} = O(n^{-3/4})\,.
\end{align*}
Combining this with~\eqref{eq-xt-xt*-distance} via the triangle inequality implies the statement of the lemma with the estimate $O(n^{-3/4}) < n^{-2/3}$ for sufficiently large $n$.
\end{proof}

In the course of the proof above we obtained in~\eqref{eq-dynamics-mixing} that for any $\delta > 0$,
\begin{equation}\label{e:tmixUpperBound}
\tmix(\delta) \leq t_0 = 2\sobinf^{-1}\log n\,.
\end{equation}
As a consequence of this, we claim that in our proof of Theorem~\ref{mainthm-arbitrary} we may assume without loss of generality that for large enough $n$ we have
\begin{align}
  \label{eq-Delta-sob-rho-assumption}
  \rho &\leq n^{1/16}\,,\quad\mbox{ and }\quad
  \Delta/\sobinf \leq \tfrac{1}{100}\log n\,.
\end{align}
Indeed, recall that our aim is to show that $\tmix(\delta)-\tmix(1-\delta)< 16\Delta \sobinf^{-2}\log \rho$ for any fixed $\delta>0$, and observe that to this end we need only consider the case where $\tmix(\delta) \geq 16\Delta\sobinf^{-2}\log\rho$ for small enough $\delta>0$ as otherwise the sought statement trivially holds. Combining this with the aforementioned upper bound on $\tmix(\delta)$
gives
\begin{equation}
  \label{eq-Delta-sob-logrho-upper}
  \Delta \sobinf^{-1}\log\rho \leq \tfrac1{8} \log n\,,
\end{equation}
producing the above bounds on $\rho$ (as $\Delta \geq 2$) and $\Delta/\sobinf$ (as $\rho \to \infty$ with $n$).

Armed with these estimates as well as Lemma~\ref{lem-sparse-product} we now turn to prove Theorem~\ref{thm-tv-project}.
\begin{proof}[\textbf{\emph{Proof of Theorem~\ref{thm-tv-project}}}]
By Lemma~\ref{lem-sparse-product} it suffices to show that the product chain $(X_t^*)$ with coordinates $(X_t^i)$ and stationary measure $\pi = \prod \pi_i$ satisfies
\begin{equation}
  \label{eq-Xt^*-tv-asymp}
  \left\|\P_{\sigma_0^*}(X^*_t(\Lambda) \in\cdot)-\pi|_\Lambda \right\|_\tv = 2\Phi\Big(\sqrt{\ltwo_t(\Lambda,\sigma_0^*)}/2\Big)-1 + o(1)
\end{equation}
for any sparse set $\Lambda$, and initial configuration $\sigma_0^*$ and any $\frac14 \log n \leq t \leq t_0$.

Consider the chain $(X_s^i)$ for some $1\leq i\leq L$, run from an arbitrary initial configuration $\sigma_0$ till time
\[ t \geq s_0 \deq 4\sobinf^{-1}\log\rho\,.\]
It is easily seen that for any two measures $\varphi,\psi$ on a state space $\Gamma$ and a map $f : \Gamma \to \Gamma'$, the $L^\infty(\pi)$ distance of $\varphi$ from $\psi$ can only decrease when taking the marginal of the measures on $\Gamma'$ since
\begin{align*}
\max_{y\in \Gamma'} \left|\frac{\varphi(f^{-1}(y))}{\psi(f^{-1}(y))} - 1 \right|
&\leq \max_{y\in\Gamma'}
\sum_{x \in f^{-1}(y)} \frac{\|\varphi-\psi\|_{L^\infty(\psi)}\;\psi(x)}{\psi(f^{-1}(y))}
= \|\varphi-\psi\|_{L^\infty(\psi)}\,.
\end{align*}
Specialized to our setting we get that
\begin{align*}
\left\|\P_{\sigma_0}(X_t^i(A_i)\in\cdot) - \pi_i|_{A_i} \right\|_{L^\infty(\pi_i|_{A_i})}
&\leq  \left\|\P_{\sigma_0}(X_t^i\in\cdot) - \pi_i\right\|_{L^\infty(\pi_i)} \\
&\leq
\left\|\P_{\sigma_0}(X_{t/2}^i\in\cdot) - \pi_i\right\|^2_{L^2(\pi_i)} \,,
\end{align*}
where the last transition incorporated a standard reduction from $L^\infty$ to $L^2$ (see e.g.~\cite{SaloffCoste}). Further recall that by Theorem~\ref{thm-l2-sobolev}
\begin{align*}
\left\|\P_{\sigma_0}(X_{t/2}^i\in\cdot) - \pi_i\right\|_{L^2(\pi_i)}
& \leq \exp\bigg(1-\gap\left(\frac{t}2-\frac1{4 \sob} \log\log \frac{1}{\pi_i(\sigma_0)}\right)\bigg)\,,
 \end{align*}
where $\lambda$ and $\sob$ are the spectral gap and log-Sobolev constant resp.\ of the chain $(X_t^i)$.
We can control the diameter of $A_i^+$ as follows:
\[ \diam_G(A_i^+) \leq 2r + \diam_G(A_i) \leq 20\frac{\Delta}{\sobinf}\log n + \frac12\log^2 n < \log^2 n\,,\]
where the last inequality was thanks to~\eqref{eq-Delta-sob-rho-assumption}. As a consequence, we have $\sob \geq \sobinf $ by definition of $\sobinf$.
Moreover, since $|A_i| \leq \rho^3 \log n$ we have
\begin{align*}
\log\log \frac{1}{\pi_i(\sigma_0)} &= (1+o(1))\log |A_i^+| \leq (1+o(1)) \log(\rho |A_i|) \\
&\leq  (1+o(1)) \log(\rho^4 \log n) \leq (5+o(1))\log\rho\,,
\end{align*}
(here we used the fact that $\rho \geq \log n$) and the fact $t \geq s_0$ now implies that
\[ \frac{t}2-\frac1{4 \sob} \log\log \frac{1}{\pi_i(\sigma_0)} \geq (\tfrac34-o(1)) \gap^{-1} \log \rho\,,\]
hence for sufficiently large $n$ we get
\begin{equation}\label{e:LInftyBoundBlock}
\left\|\P_{\sigma_0}(X_t^i(A_i)\in\cdot) - \pi_i|_{A_i} \right\|_{L^\infty(\pi_i|_{A_i})}
\leq \rho^{-3/2+o(1)} = o(1)\,.
\end{equation}
By~\eqref{eq-Delta-sob-logrho-upper} we have
$ 4\sobinf^{-1}\log \rho \leq \frac1{2\Delta}\log n \leq \frac14\log n $ (the last inequality due to the fact that $\Delta\geq 2$),
and so in our setting indeed $t \geq \frac14 \log n \geq s_0$.  We have thus established that every $(X_t^i)$
fulfills the requirement~\eqref{eq-product-assumption} of Proposition~\ref{prop-l1-l2} for arbitrarily small $\epsilon > 0$.
The application of this proposition now establishes~\eqref{eq-Xt^*-tv-asymp}, completing the proof.
\end{proof}

\subsection{Proof of Theorem~\ref{mainthm-arbitrary}}
We begin the proof by formulating an expression for the cutoff location.
Let $\ltwo_t$ be as in Theorem~\ref{thm-tv-project} and set
\begin{equation}
  \label{eq-tstar-def}
  t^\star = \inf\left\{t > 0 \;:\; \max_{\Lambda\in\sparse}\max_{\sigma_0} \ltwo_t(\Lambda,\sigma_0) \leq \rho \right\}\,.
\end{equation}
Observe that $\ltwo_t(\Lambda,\sigma_0)$ is continuous and monotone decreasing in $t$. A useful lower bound on these quantities is the following:
\begin{equation}\label{eq-ltwot-lower}
\max_{\Lambda\in\sparse}\max_{\sigma_0}\ltwo_t(\Lambda,\sigma_0) \geq n\rho^{-3}\mathrm{e}^{-2t}\quad\mbox{for any $t\geq 0$}\,.
\end{equation}
To see this, let $\Lambda$ consist of an arbitrary set of $n\rho^{-3}$ singletons whose pairwise distances all exceed $3r$, the existence of which is guaranteed by the definition of $\rho$. It therefore suffices to show that for each $v \in \Lambda$ we have $\ltwo_t(\{v\},\sigma_0) \geq \exp(-2t)$ for an appropriately chosen $\sigma_0 \in \{\pm1\}^{B_v(r)}$.
Write $\psi = \mu_{B_v(r)}^\emptyset|_{\{v\}}$, assume without loss of generality that $\psi(-1)\geq\frac12$ and let $\sigma_0$ be the all-plus starting configuration. Write $\varphi_t = \P_{\sigma_0}(X^*_t(v)\in\cdot)$ and note that the monotonicity of the Glauber dynamics implies that $\varphi_t$ stochastically dominates $\psi$ for any $t\geq 0$, and furthermore this holds even when conditioning on the sequence of updates sites up to time $t$ (yet without revealing the unit-variables used for generating the new spins). As such, if $\tau_v = \inf\{ t : \mbox{ $X^*_t$ updates the site $v$}\}$ then for any $t \geq 0$
\[ \varphi_t(1) \geq \P(\tau_v > t) + \psi(1)\P(\tau_v \leq t) = \psi(1) + \mathrm{e}^{-t}\psi(-1) \geq \psi(1)+\frac12\mathrm{e}^{-t}\,,\]
with the last inequality due to our assumption on $\psi(-1)$. Immediately it follows that
\begin{align*}
\|\varphi_t - \psi\|_{L^2(\psi)} &\geq \|\varphi_t - \psi\|_{L^1(\psi)} = 2\left|\varphi_t(1)-\psi(1)\right| \geq \mathrm{e}^{-t}\,,
\end{align*}
thus establishing~\eqref{eq-ltwot-lower}.
Since $\rho \leq n^{1/16}$ thanks to~\eqref{eq-Delta-sob-logrho-upper} we immediately observe that at $t=0$ we have $\max_{\Lambda}\max_{\sigma_0}\ltwo_t(\Lambda) \geq n\rho^{-3} > \rho$ and hence it follows from continuity that
\[ \max_{\Lambda\in\sparse}\max_{\sigma_0} \ltwo_{t^\star}(\Lambda,\sigma_0) = \rho \,.\]
Moreover, Eq.~\eqref{eq-ltwot-lower} implies that $\exp(-2t^\star) \leq \rho^4 / n \leq n^{-3/4}$ (where we again used the fact that $\rho \leq n^{1/16}$) and after rearranging this yields
\begin{align}
  \label{eq-tstar-lower}
  t^\star \geq \frac{3}8\log n > \frac14 \log n\,.
\end{align}
Conversely, using the log-Sobolev argument as in the proof of Theorem~\ref{thm-tv-project} one obtains that if $\Lambda$ is a sparse set and $A_1,\ldots,A_L$ are the components comprising its partition then, using the same notation as Theorem~\ref{thm-tv-project} we have
\[ \max_{\sigma_0}\lltwo_{t_0}(A_i,\sigma_0) \leq \exp\left[-2\left(t_0 + (\tfrac54 + o(1))\log \rho\right)\right] \leq n^{-3/2}\,.\]
As $L\leq n$ it follows that for $t\geq t_0$ one has $\max_{\Lambda} \max_{\sigma_0}\ltwo_{t}(\Lambda,\sigma_0) \leq n^{-1/2}$ whereas at $t=t^\star$ this quantity should equal $\rho \geq \log n$. Altogether we deduce that for large enough $n$
\begin{align}
  \label{eq-tstar-upper}
  t^\star < t_0 = 2 \sobinf^{-1}\log n\,.
\end{align}

At this point it we can establish that $\tmix(1-\epsilon) \geq t^\star$ for any fixed $\epsilon > 0$ and large enough $n$.
Indeed, the above given bounds on $t^\star$ place it in the admissible range for an application of Theorem~\ref{thm-tv-project}, and
taking $\Lambda$ and $\sigma_0$ to be those achieving the maximum in the definition~\eqref{eq-tstar-def} we get that
\begin{align*}
\left\|\P_{\sigma_0}(X_{t^\star}\in\cdot)-\mu \right\|_\tv &\geq
\left\|\P_{\sigma_0}(X_{t^\star}(\Lambda)\in\cdot)-\mu|_{\Lambda} \right\|_\tv \\
&= 2\Phi(\sqrt{\ltwo_t(\Lambda,\sigma_0)}/2)-1 - o(1)\,.
\end{align*}
Since $\ltwo_{t^\star}(\Lambda,\sigma_0) =\rho \to \infty$ we infer that $\left\|\P_{\sigma_0}(X_{t^\star}\in\cdot)-\mu \right\|_\tv =1-o(1)$, as required.

It remains to show that $\tmix(\epsilon) \leq t^\star + 16 \Delta \sobinf^{-2}\log\rho$ for any fixed $\epsilon>0$ and large enough $n$.
Thanks to the reduction provided by Theorem~\ref{thm-xt-bound-a-sparse}, this would follow from showing that
\begin{align}\label{eq-t-star-int-upper}
 \max_{\sigma_0}\int_{\sparse}\left\|\P_{\sigma_0}(X_{t^\star}(\Lambda)\in\cdot)-\mu|_{\Lambda} \right\|_\tv\, d\nu(\Lambda) = o(1)
\end{align}
where $\nu$ is the distribution on $\sparse$ specified in that theorem.
By Theorem~\ref{thm-tv-project} and the bounds~\eqref{eq-tstar-lower},\eqref{eq-tstar-upper} on $t^\star$ we can replace the integrand in the above equation by $2\Phi(\sqrt{\ltwo_{t^\star}(\Lambda,\sigma_0)}/2)-1 - o(1)$ and hence reduce the task of establishing~\eqref{eq-t-star-int-upper} to showing that
\begin{align}\label{eq-t-star-int-upper2}
\max_{\sigma_0}\int_{\sparse}
\Phi(\sqrt{\ltwo_{t^\star}(\Lambda,\sigma_0)}/2)\, d\nu(\Lambda) = \frac12 + o(1)\,.
\end{align}
To this end, let $\sigma_0$ be arbitrary and let $\Lambda_0$ be the sparse set that maximizes $\ltwo_{t^\star}(\Lambda_0,\sigma_0)$. To simplify the exposition, in what follows we will omit the reference to $\sigma_0$ as an argument of $\ltwo_{t^\star}$ when there is no danger of confusion.

It will be useful to define another sparse set $\widetilde{\Lambda}$ as follows. Begin by setting $\widetilde{\Lambda}=\emptyset$, then repeatedly add components $\{\tilde{A}_i\}$ to $\widetilde{\Lambda}$ via the following rule:
\begin{itemize}
  \item Let $\tilde{A}_1,\ldots,\tilde{A}_{k-1}$ be the components already collected.
  \item Let $\tilde{A}_{k}\subset V$ be the subset maximizing $\ltwo_{t^\star}(\tilde{A}_k)$
  subject to the constraint $\dist_G(\tilde{A}_i,\tilde{A}_k) \geq 3r$ for all $i < k$ as well as $|\tilde{A}_k|\leq\rho^3\log n$ and $\diam_G(\tilde{A}_k) \leq \frac12\log^2 n$.
  \item The process terminates once no such set $\tilde{A}_k$ exists.
\end{itemize}
Note that clearly
\begin{align}\label{eq-tildeLambda-Lambda0}
\ltwo_{t^\star}(\widetilde{\Lambda}) \leq \ltwo_{t^\star}(\Lambda_0) \leq \rho
\end{align}
(the last inequality may be strict since $\sigma_0$ was chosen arbitrarily as opposed to selecting the configuration for which $\ltwo_{t^\star}$ could obtain its global optimum).
For each $v\in V$ we associate the following quantity measuring the weight of the closest $\tilde{A}_i$ to it:
\[ \theta_v = \max\left\{\lltwo_{t^\star}(\tilde{A}_i)\;:\; \dist_G(v,\tilde{A}_i) \leq 3r\right\}\,.\]
Motivating this definition is the following simple fact: We argue that if $A_i$ is a component of some sparse set $\Lambda\in\sparse$ then
\[
   \lltwo_{t^\star}(A_i) \leq \max_{v \in A_i} \theta_v\,.
\]
To verify this fact, let $v$ be the vertex with the maximal $\theta_v$ among all vertices of $A_i$ and let $k\geq 1$ be the smallest index such that $\dist_G(v,\tilde{A}_k) \leq 3r$. (Note that $k$ is well-defined since if no such $\tilde{A}_k$ existed then $v$ itself would be admissible as an additional singleton cluster in $\widetilde{\Lambda}$, contradiction.) By the maximality of $v$ that is to say that $k$ is the smallest index such that $\dist(A_i,\tilde{A}_k) \leq 3r$.
Since $\dist_G(A_i,\tilde{A}_j)>3r$ for all $j<k$ it follows that $A_i$ satisfies all the conditions required from $\tilde{A}_{k+1}$ and thus by definition of $\widetilde{\Lambda}$
\[ \lltwo_{t^\star}(A_i) \leq \lltwo_{t^\star}(\tilde{A}_{k+1}) \leq \lltwo_{t^\star}(\tilde{A}_k) = \theta_v\,,\]
as claimed. Since the $A_i$'s are disjoint, this fact implies in particular that
\[
\ltwo_{t^\star}(\Lambda) \leq \sum_{v \in V} \theta_v \one_{\{v \in \Lambda\}}\quad\mbox{for any $\Lambda\in\sparse$}\,.
\]
Recall now that the measure $\nu$ over the collection of sparse sets $\sparse$ satisfies $\nu(\Lambda : v\in\Lambda) \leq \rho^{-10}$ for any $v \in V$. Plugging this in the above inequality we obtain that
\begin{align}\label{eq-thetav-ineq}
\int_\sparse \ltwo_{t^\star}(\Lambda) \, d\nu(\Lambda) \leq
 \rho^{-10} \sum_{v \in V} \theta_v \,.
 \end{align}
To estimate $\sum_v \theta_v$ recall that each $\theta_v$ amounts to
$\ltwo_{t^\star}(\tilde{A}_i)$ for some $\tilde{A}_i$ whose distance from $v$ is at most $3r$. Thus, each $\ltwo_{t^\star}(\tilde{A}_i)$ appears at most $|\tilde{A}_i| \rho^3$ times in this sum, and since by construction $|\tilde{A}_i| \leq \rho^3\log n$ we obtain that
\begin{align*}
\sum_v \theta_v &\leq \sum_i\lltwo_{t^\star}(\tilde{A}_i) \rho^6 \log n = \ltwo_{t^\star}(\widetilde{\Lambda}) \rho^7 \log n
\leq \ltwo_{t^\star}(\Lambda_0) \rho^7 \log n \leq \rho^{8}\log n\,,
 \end{align*}
where the last two inequalities used~\eqref{eq-tildeLambda-Lambda0}.
By combining this with~\eqref{eq-thetav-ineq} (and recalling that $\rho \geq \log n$) we now infer that
\[ \int_\sparse \ltwo_{t^\star}(\Lambda)\, d\nu(\Lambda) \leq \rho^{-1} = o(1)\,.\]
 Markov's inequality now implies that $\nu(\{\Lambda : \ltwo_{t^\star}(\Lambda) \geq \epsilon\}) = o(1)$ for any fixed $\epsilon>0$. This readily implies Eq.~\eqref{eq-t-star-int-upper2} by the continuity of the integrand in that equation
 as a function of $\ltwo_{t^\star}$ and thus completes the proof. \qed

\section{General spin system models}\label{sec:general}

In this section we extend the cutoff criterion, established in the previous sections for the Ising model, to general spin systems as defined in \S\ref{sec:prelim-spin-systems}, including for instance the Potts, proper coloring and hard-core models.  

While to a large extent the proofs presented thus far were not model specific and mainly used bounds on the log-Sobolev constant, we did use the monotonicity of the Ising model in an essential way in the context of the barrier dynamics operator $\cG_s$. Indeed, as defined $\cG_s$ is a \emph{grand coupling}, a coupling of all initial configurations simultaneously.  The monotonicity of the Ising model implies that if the chains starting from all-plus and all-minus agree at some vertex $v$ then so do those started at any other configuration.  For general spin-systems which do not have monotonicity we will construct a grand coupling and analyze it using ideas similar to those in the construction of perfect simulation algorithms (see, e.g.,~\cite{Huber}).

We will consider Markovian grand couplings where the Poisson clocks on the vertices are identical for all chains and the updates depend only on the current states of each chain and new independent randomness.

For such a grand coupling, let $X_t^{x_0}$ denote  the copy of the chain started from $X_0=x_0$.
Further define the disagreement process $\sX: \mathbb{R}_+ \times \Lambda \to \{0,1\}$ with $\sX_t(u)$ denoting the indicator that  there exist two initial configurations $x_0,x_0' \in \spinset^{\Lambda}$ such that at time  $t$ the corresponding chains disagree at site $u$, that is, $X_t^{x_0}(u)\neq X_t^{x_0'}(u)$.  
By this definition clearly $\sX_0\equiv 1$.  If for some $C,c>0$ the coupling satisfies the following exponential coupling condition
\begin{equation}\label{e:exponentialCoupling}
\max_{u\in\Lambda}\P(\sX_t(u)=1) \leq C e^{-ct},
\end{equation}
then we attain the following generalization of our framework for the Ising model to the case of a general spin-system.
\begin{theorem}\label{t:general}
Fix $d\geq 1$ and suppose that there exists a Markovian grand coupling for Glauber dynamics for a general spin-system on a box $\Lambda \subset \Z^d$ of side-length $n$ which satisfies~\eqref{e:exponentialCoupling}. Then the dynamics exhibits cutoff with a window of $O(\log\log n)$.
\end{theorem}

\begin{proof}
We describe the modifications required in the proof of Theorem~\ref{mainthm-arbitrary} to complete the result. First note that the bounds of Lemma~\ref{l:contact} imply order $\log n$ mixing under any boundary conditions and so combined with Theorem~2.3 of~\cite{DSVW} we have strong spatial mixing and hence a uniformly bounded log-Sobolev constant~\cite{MO2}.
Most of the proof and the preceding lemmas can then be taken as unchanged yet in a few places monotonicity is exploited.
We define the barrier-dynamics in the same way, taking the single spin updates according to the grand coupling described above.

In Lemma~\ref{lem-sparse-prob} monotonicity is used to give an upper bound on the probability of all chains agreeing at a vertex $v$ via the probability of the all plus and minus agreeing at $v$.  The exponential coupling bound~\eqref{e:exponentialCoupling} gives the required  bounds in the general setup.  In particular, taking $U_v$ to denote the event that there exist two initial conditions such that the outputs of the barrier-dynamics operator $\mathcal{G}_{C\log\log n}$ disagree on the value of $v$, we get that $\P(U_v) = O(\log^{-C' d} n) \leq  \rho^{-6\Delta / \sobinf}$ by~\eqref{e:exponentialCoupling}, thereby establishing the analogue of equation~\eqref{eq-Uv-bound}.
Similarly, in Lemma~\ref{thm-tv-project} a point-wise monotone coupling bound is used to bound the total variation between two measures which can be replaced by the exponential grand coupling bound.

The proof of Theorem~\ref{mainthm-arbitrary} begins with an estimate showing that $t^\star \geq \frac38\log n$ which ensures that the dynamics on the components $\tilde{A}_i$ are well mixed in the $L^\infty$ distance.  On the lattice $\Z^d$ all such components are poly-logarithmic in size and so mix within time $O(\log\log n)$.  We may instead take
\begin{equation*}
  t^\star = \inf\left\{t > (\log\log n)^2 \;:\; \max_{\Lambda\in\sparse}\max_{\sigma_0} \ltwo_t(\Lambda,\sigma_0) \leq \log\log n \right\}\,.
\end{equation*}
which is sufficiently large to apply the argument in the proof of Theorem~\ref{mainthm-arbitrary} and so implies an upper bound on the mixing time of \begin{equation}\label{e:modifiedTstar}
\tmix(\epsilon) \leq t^\star + C\log \log n \, ,
\end{equation}
for large enough  $C>0$.  Results of~\cite{HaSi} imply that $\tmix(\epsilon) > c \log n$ for some $c>0$ and so $t^\star > c\log n - C\log \log n$. Since we have excluded the possibility that $\tmix(1-\epsilon) < (\log\log n)^2$ for large $n$ it follows from the proof of Theorem~\ref{mainthm-arbitrary} that $\tmix(1-\epsilon) \geq t^\star$ for large enough $n$.  Combined with~\eqref{e:modifiedTstar} this establishes cutoff with a window of order $\log\log n$.
\end{proof}

\subsection{Soft interactions}

We now construct grand couplings satisfying~\eqref{e:exponentialCoupling} when the temperature is sufficiently high.  We will focus on the heat-bath dynamics but note that analogous results apply for the Metropolis-Hastings chain (as detailed in the remark at the end of this section).

For a general spin system on a box $\Lambda\subset\Z^d$ with a state space $\spinset^\Lambda$ for some finite set $\spinset$ (see~\S\ref{sec:prelim-spin-systems} for definitions) define the following measure of ``temperature'' for the system:
\[\zeta_\mu = \sum_{s \in \spinset} \min_{\substack{x\in \Lambda \\ \eta\in\spinset^\Lambda}} \mu\Big(\sigma(x) = s \given \sigma(\Lambda\setminus\{x\})=\eta(\Lambda\setminus\{x\})\Big)\,.\]
It is easy to see that $0 \leq \zeta_\mu \leq 1$.
Roughly stated, values of $\zeta_\mu$  close to 1 correspond to systems with weak interactions and high temperatures.  

Given $\zeta_\mu$ one can define a grand coupling for heat-bath Glauber dynamics as follows.
Recall that a vertex $x$ is chosen to be updated by a rate-one Poisson process and with the new state being given as some function of $\sigma(\Lambda\setminus\{x\})$ and a random variable $U$ uniform on $[0,1]$.  When $U\leq \zeta_\mu$ we can set $\sigma(x)$ independently of the current configuration, selecting state $s$ with probability
\[
\min_{\eta \in \spinset^{\Lambda}} \mu\Big(\sigma(x) = s \given \sigma(V\setminus\{x\})=\eta(\Lambda\setminus\{x\})\Big)\,.
\]
For $U > \zeta_\mu$ the transition probabilities can be selected arbitrarily so long as the chain has the correct marginals.

The \emph{contact process} is a well known stochastic process which models the spread of infections (see e.g.~\cite{Liggett}).
 Denoted here by $\sY: \mathbb{R}_+ \times \Lambda \to \{0,1\}$, a vertex in state  1 corresponds to an infection which heals over time but may also spread to neighboring vertices.  More formally, vertices are set to 0 (healed) according to i.i.d.\ rate $h$ Poisson clocks.    Infections spread along edges at according to rate $\lambda$ Poisson clocks so that if one end is 1 when the clock rings, the other end is set to 1.  The following lemma bounds the disagreement process by a contact process and shows that under the following condition
\begin{equation}\label{e:zetaCondition}
\liminf_{n\to\infty} \zeta_\mu > \frac{2d}{2d+1},
\end{equation}
the probability of a disagreement converges to 0 exponentially fast.

\begin{lemma}\label{l:contact}
The disagreement process $\sX$ is stochastically dominated by a contact process $\sY$ with healing rate $\zeta_\mu$ and infection rate $1-\zeta_\mu$.  Then 
\begin{equation}\label{e:grandCoupling bound}
\max_{u \in \Lambda} \P(\sX_t(u)=1)\leq \exp\left(- \left(2d+1\right)\left(\zeta_\mu -\frac{2d}{2d+1}\right) t\right).
\end{equation}
\end{lemma}

\begin{proof}
If a vertex $u$ with $\sX_t(u)=1$ is updated then by construction with probability at least $\zeta_\mu$ the grand coupling selects the same spin in each the chains. That is, we set $\sX_t(u)=0$ with probability exceeding the healing rate of the contact process.
At the same time, a new disagreement can be created only when updating a vertex adjacent to a vertex with a disagreement.  In this case the probability of a new disagreement is at most $1-\zeta_\mu$ and hence this bounds the rate of the spread of infections.  Taken together this implies that $\sY_t$ stochastically dominates $\sX_t$. 

Define the number of infected sites as $W_t= \sum_{u\in\Lambda} \sY_t(u)$.   Vertices are healed at a rate of $W_t \zeta_\mu$ while new infections occur at a rate of at most $2d(1-\zeta_\mu) W_t$ which may be less when infected vertices are adjacent.  Hence
\begin{align*}
\frac{d \E W_t}{dt} \leq \E W_t (-\zeta_\mu + 2d(1-\zeta_\mu) ) = - (2d+1)\left(\zeta_\mu -\frac{2d}{2d+1}\right)\E W_t \, ,
\end{align*}
and so
\begin{equation}\label{e:infectionBound}
\E W_t \leq |\Lambda| \exp\left(- (2d+1)\left(\zeta_\mu -\frac{2d}{2d+1}\right) t\right) \, ,
\end{equation}
which is only meaningful when equation~\eqref{e:zetaCondition} holds.  If $\Lambda$ is transitive, for example a torus, then it would immediately follow that by symmetry that
\begin{equation}\label{e:infectionVertexBound}
\P(\sY_t(u)=1) \leq \exp\left(- (2d+1)\left(\zeta_\mu -\frac{2d}{2d+1}\right) t\right).
\end{equation}
For general $\Lambda \subset \Z^d$ we can embed $\Lambda$ into a large torus and note that the contact process is monotone in the edge set of the graph implying~\eqref{e:infectionVertexBound} for general graphs.  The stochastic domination then implies equation~\eqref{e:grandCoupling bound}, completing the lemma.
\end{proof}

With this construction we are now able to complete the generalization to high temperature systems with soft interactions.
 It is easy to verify that on any graph with maximal degree $\Delta$ satisfies $\zeta_\mu > \Delta/(\Delta+1)$ for
  \begin{compactitem}[\indent$\bullet$]
    \item ferromagnetic Potts model with $0 \leq \beta < \frac1{2\Delta}\log(1+\frac{q}{\Delta})$.
    \item anti-ferromagnetic Potts model with $0 > \beta > -\frac1{2\Delta}\log(1+\frac{q}{\Delta(q-1)})$.
    \item gas hard-core model with $\lambda < 1/\Delta$.
  \end{compactitem}
 A more direct analysis as in~\cite{HuberThesis}*{Section~5.2},\cite{Huber}*{Theorem~3} and~\cites{LV,Vigoda} (noting that in the last case the coupling can be extended to a suitable grand coupling as required) establishes the exponential coupling condition~\eqref{e:exponentialCoupling} in a somewhat broader regime for the ferromagnetic Potts, anti-ferromagnetic Potts and hard-core models, resp. Specifically, in the ferromagnetic and anti-ferromagnetic Potts models the requirements become $0\leq \beta < \frac12\log(\frac{\Delta}{\Delta-1})$ and $0 > \beta \geq -\frac12 \log(\frac{\Delta q}{\Delta q-1})$ resp., whereas 
 in the hard-core model the requirement on the fugacity becomes $\lambda<2/(\Delta-2)$. Combining these results with Theorem~\ref{t:general} establishes Theorem~\ref{mainthm-potts-Zd}. In the special case of the square lattice $\Z^d$ we can
refine the requirement $\lambda<1/(d-1)$ by exploiting the bipartite nature of the geometry. Indeed, the standard reordering of the partial order on the lattice turns the system into monotone, at which point the analysis of \S\ref{sec:arbitrary} becomes valid. Theorem~\ref{mainthm-hardcore-Zd} then follows from the work of Weitz~\cite{Weitz} which established strong spatial mixing whenever $\lambda < \frac{(\Delta-1)^{\Delta-1}}{(\Delta-2)^\Delta}$.

\begin{remark*}
The above analysis for the heat-bath dynamics can be easily extended to the Metropolis-Hastings chain.  The main alteration is to modify the definition of $\zeta_\mu$ and set
\[\zeta_\mu' = |\spinset|^{-1}\sum_{s \in \spinset} \min_{\substack{x\in \Lambda \\ s'\in\spinset\\ \eta\in\spinset^\Lambda}} \frac{\mu\Big(\sigma(x) = s \given \sigma(\Lambda\setminus\{x\})=\eta(\Lambda\setminus\{x\})\Big)}{\mu\Big(\sigma(x) = s' \given \sigma(\Lambda\setminus\{x\})=\eta(\Lambda\setminus\{x\})\Big)}\,.\]
Similarly to the heat-bath analysis, it is easy to verify that $0\leq \zeta_\mu' \leq 1$ and that one can construct a Markovian grand coupling for Metropolis where
each update is independent of the present state with probability $\zeta'_\mu$. 
With this definition we therefore again obtain that $\inf_n \zeta'_\mu > 2d/(2d+1)$ is a sufficient condition for cutoff on $\Z^d$.
\end{remark*}

\subsection{Proper colorings}
In the case of proper $q$-colorings condition~\eqref{e:zetaCondition} never holds
as we always have that $\zeta_\mu=0$ since having a neighbor with color $s$ precludes a vertex from having color $s$.  We instead use a modified grand coupling as defined in~\cite{Huber}*{Section~5}.  When a vertex $u$ is selected for updating we generate a random permutation of the $q$ colors and choose the first color which does not appear amongst its neighbors.  This way, in the setting of the lattice $\Z^d$ one of the first $2d+1$ colors must be chosen regardless of the local neighborhood.  Now~\cite{Huber}*{Theorem~2} and its proof imply that with $q\geq 4d(d+1)$ colors the probability of a disagreement in the grand coupling decays exponentially fast, thus establishing equation~\eqref{e:exponentialCoupling}. The proof of Theorem~\ref{mainthm-coloring-Zd} then follows from Theorem~\ref{t:general}.

\section{Cutoff for lattices with boundary conditions}\label{sec:cutoff-locations}

In this section we analyze the important case of the Glauber dynamics $(X_t)$ on the cube $\Z_n^d$ with plus boundary conditions with a proof that naturally extends to the setting of free boundary conditions as well.  Whilst boxes with periodic boundary conditions are transitive, imposing plus boundary conditions  means that vertices in different parts of the graph must be treated differently.  We capture this by characterizing vertices in terms of how many faces of the cube they are close to.

As well as showing cutoff we will determine the cutoff location in terms of spectral gaps for infinite volume dynamics.  Define $\HS_j = \HS_{j,d}$ to be the intersections of half-planes in $\Z^d$ given by $\Z_+^j \times \Z^{d-j}$ where $\Z_+=\{x\in\Z:x>0\}$. We let $\lambda_j^{(\infty)}=\lambda_{j,d}^{(\infty)}$ denote the spectral gap of the Glauber dynamics on $\HS_{j}$ with all-plus boundary conditions and let $\mu_j$ denote its stationary distribution.  With this definition we prove the following theorem to which Theorem~\ref{mainthm-plusBC2d} is a special case.
\begin{theorem}\label{t:plusGeneral}
The Glauber dynamics $(X_t)$ for the Ising model on $\Z_n^d$, the cube of side-length $n$ with all-plus boundary conditions exhibits cutoff at
\[
\frac12 \max_{0\leq j < d} (d-j) \lambda_{j,d}^{(\infty)}\log n\]
with a cutoff window of width $O(\log \log n)$.
\end{theorem}

First we begin by defining blocks with a mixture of plus and periodic boundary conditions.  Fix $m=\lfloor \log^3 n \rfloor$ and for each $0\leq j \leq d$ let $\block^{(j)}$ be the graph $\Z_m^d$ where we impose plus boundary conditions on the first $j$ coordinates and periodic boundary conditions on the remaining $d-j$ coordinates.  In the case of $d=2$ then $\block^{(0)}$ is the torus $(\Z/m\Z)^2$, $\block^{(1)}$ corresponds to a cylinder with plus boundaries at then ends while $\block^{(2)}$ corresponds to a box with all plus boundary conditions.

Let $B=B_j$ denote the subset of $\block^{(j)}$ given by
\[
B_j = \left\{1,\ldots,\frac{2m}{3}\right\}^{j}\times\left\{\frac{m}{6},\ldots,\frac{5m}{6}\right\}^{d-j}~,
\]
which we will refer to as the inner block.  Note that the inner blocks $B_j$ touch faces of the coordinates where there is a plus boundary condition.  We define
\begin{align}
  \label{eq-mjt-def}
  \lltwo^{(j)}_t &= \lltwo^{(j)}_t(m) := \max_{x_0} \left\| \P_{x_0}\big(X^{*,j}_t(B) \in \cdot \big)  - \mu^{*,j}|_{B} \right\|^2_{L^2(\mu^*|_{B})}\,,
\end{align}
where $X^{*,j}_t$ denotes the Glauber dynamics on $\block^{(j)}$ and $\mu^{*,j}$ its stationary distribution.

For $i\in \Z_n$ we define the following intervals:
\[
\bar J_i=\begin{cases} \{1,\ldots,m\} & 1\leq i \leq \frac{m}{2},\\
\{i-\frac{m}{2},\ldots,i+\frac{m}{2}-1\} & \frac{m}{2} +1 \leq i \leq n-\frac{m}{2},\\
\{n-m+1,\ldots,n\} & n-\frac{m}{2} + 1 \leq i \leq n,
\end{cases}
\]
and
\[
J_i=\begin{cases} \{1,\ldots,\frac{2m}{3}\} & 1\leq i \leq \frac{m}{2},\\
\{i-\frac{m}{6},\ldots,i+\frac{m}{6}-1\} & \frac{m}{2} +1 \leq i \leq n-\frac{m}{2},\\
\{n-\frac{2m}{3}+1,\ldots,n\} & n-\frac{m}{2} + 1 \leq i \leq n.
\end{cases}
\]
For each $x=(x_1,\ldots,x_d)\in \Z_n^d$ let $\bar J_x$ denote the block
\[
\bar J_{x_1}\times \ldots,\times \bar J_{x_d}~,
\]
and define the inner block $J_x$ similarly.  We classify vertices by the number of sides they are close to via $\varphi(x)\deq\#\{1\leq i \leq d: \max\{x_i,n+1-x_i\} \leq \frac{m }{2}\}$.  Then for each $x\in \Z_n^d$ there exists a graph isomorphism $\psi_x$ from $\bar J_x$ into $\block^{(\varphi(x))}$ such that $J_x$ is mapped to $B\subset \block^{(\varphi(x))}$.  It can easily be checked that for each $0\leq j \leq d$ and large $n$,
\begin{equation}\label{e:vertexTypeSize}
\frac12 m^j n^{d-j}  \leq \mathcal{N}_j \leq 2^d m^j n^{d-j}~,
\end{equation}
where $\mathcal{N}_j=\left| \{x:\varphi(x)=j \} \right|$ denotes the number of vertices of type $j$.

By Theorem~\ref{thm-log-sobolev-torus} the dynamics on all relevant rectangles has log-Sobolev constant at least $\sobinf$.

\subsection{Upper bound on the $L^1$-distance}\label{sec:plusUBound}

With the framework of \S\ref{sec:arbitrary} as our starting point, define the barrier dynamics as before except we replace balls in the graph distance with balls in the $L^\infty$-distance on the lattices as these are rectangles whose log-Sobolev constant can be bounded using Theorem~\ref{thm-log-sobolev-torus}.  With this minor modification of Lemma~\ref{lem-sparse-prob} we get the following lemma.

\begin{lemma}\label{lem-sparse-prob-plus}
Let $\mathcal{G}_s$ be the barrier-dynamics operator on a $d$-dimensional rectangle with side-lengths at most $n$ and any combination of arbitrary and periodic boundary conditions.  Let $W_s$ be its update sequence up to time $s$ for $s = 100 d^2 \sobinf^{-2}\log \log n$ and let $\sparse$ be the collection of sparse sets of $G$.
For sufficiently large $n$ we have $\P(\Lambda_{W_s} \in \sparse) \geq 1- n^{-10d}$ and furthermore $\P(v\in\Lambda_{W_s}) \leq \log^{-10d} n$ for every $v\in V$.
\end{lemma}

Let $\Lambda\in \sparse$ be a sparse subset of $\Z_n^d$ partitioned into $L=L(\Lambda)$ components $\{A_i\}_{i=1}^L$ according to Definition~\ref{def-sparse-set}.  For each $A_i$ let $y_i$ be a representative element chosen according to any arbitrary rule.  We will classify the components according to the value of $\varphi(y_i)$ and define
\[
L_j = \#\{i:\varphi(y_i) = j\}.
\]
By the diameter bound on the components $A_i$ we have that  $A_i\subset J_{y_i}$.

For each $i$ let $\block_i$ be a copy of $\block^{(\varphi(y_i))}$ and let $G^*$ denote the graph given by the union of these components, all disconnected from each other inheriting the boundary conditions of the $\block_i$.  By a slight abuse of notation will let $\psi_i=\psi_{y_i}$ denote the map taking $\bar J_{y_i}$ into the copy $\block_i$.  By construction it also maps $J_{y_i}$ into the corresponding inner block $B_i^*$.  Finally, denote $\Lambda^*$ as the image of $\Lambda$.

Let $X_t^*$ denote the Glauber dynamics on $G^*$ and $\mu^*$ its stationary distribution.  We couple $X_t$ and $X_t^*$ as follows:  Let $A_i^+=B_{A_i}(r)$ where $r = \lfloor 20 d^2 \sobinf^{-1} \log n \rfloor$ as in equation~\eqref{eq-r-def}.  Note that by the sparseness properties of the $A_i$ the $A_i^+$ are disjoint.  Now we couple the two dynamics so that if $x$ is in some $A_i^+$ then we couple the updates of $x$ and $\psi_i(x)$.  For all other vertices in $\Z_n^d$ and $G^*$ we can take the updates to be independent and we set initial conditions such that
\begin{equation}\label{e:blockInitialConditions}
X_0(x)=X^*_0(\psi_i(x)) \hbox{ for all $i$ and } x\in A^+_i~.
\end{equation}
Repeating the proof of Lemma~\ref{lem-sparse-product} with minor modifications we have that
\begin{equation}\label{e:stationaryDifference}
 \Big| \left\|\P_{x_0}(X_t(\Lambda)\in\cdot)-\mu|_{\Lambda} \right\|_\tv -
 \left\|\P_{x_0^*}(X^*_t(\Lambda^*) \in\cdot)-\mu^*|_{\Lambda} \right\|_\tv \Big| < n^{-2d/3}\,
\end{equation}
and hence applying Lemma~\ref{lem-xt-bound-aw} for $t\leq 2d \sobinf^{-1}\log n$ we have that,
\begin{equation}\label{e:upperBoundPlusProduct}
 \left\|\P_{x_0}(X_{t+s}\in\cdot)-\mu \right\|_\tv \leq
\E \left\|\P_{x_0^*}(X^*_t(\cup_i B_i^*) \in\cdot)-\mu^*|_\Lambda \right\|_\tv  + 2n^{-2d/3}\, .
\end{equation}
As $X_t^*$ is a product chain we can apply  Proposition~\ref{prop-l1-l2}  giving us that
\begin{equation}\label{e:upperBoundProductChain}
\max_{x_0^*} \left\|\P_{x_0^*}(X^*_t(\cup_i B_i^*) \in\cdot)-\mu^*|_\Lambda \right\|_\tv \leq \sqrt{\mbox{$\sum_{j=0}^d L_j \lltwo^{(j)}_t$}}\, .
\end{equation}
The randomness here is in the expected values of $L_j$ which we bound using Lemma~\ref{lem-sparse-prob-plus},
\begin{equation}\label{e:LjBound}
\E L_j \leq \sum_{v\in \Z_n^d:\varphi(v)=j} \P(v\in\Lambda_{W_s}) \leq \mathcal{N}_j \log^{-10d}n \, .
\end{equation}
Combining equations~\eqref{e:vertexTypeSize}, \eqref{e:upperBoundPlusProduct}, \eqref{e:upperBoundProductChain} and \eqref{e:LjBound} using Jensen's inequality we have that
\begin{equation}\label{e:upperBoundPlusFinal}
\max_{x_0} \left\|\P_{x_0}(X_{t+s}\in\cdot)-\mu \right\|_\tv \leq
\sqrt{\mbox{$\sum_{j=0}^d 2^d   \lltwo^{(j)}_t n^{d-j} \log^{-7d}n$}} + 2n^{-2d/3}\, .
\end{equation}

\subsection{Lower bound on the $L^1$-distance}

For a lower bound on the total variation we embed small blocks into $\Z_n^d$ of type roughly proportional to $\mathcal{N}_j$.  Taking the vertices
\[
(y_i)_{i=1,\ldots,L}:=\left\{(x_1,\ldots,x_d)\in\Z_n^d:(x_j-1)/(2m) \in\Z,1\leq j \leq d\right\} \subset \Z_n^d \, ,
 \]
the blocks $\bar J_{y_i}$ are disjoint and $L_j=\left| \{i:\varphi(y_i)=j \} \right|$, the number of vertices of type $j$, satisfies
\[
L_j \geq 3^{-d}m^{-d} n^{d-j} \, ,
\]
for large $n$.

For each $i$ let $\block_i$ be a copy of $\block^{(\varphi(y_i))}$ and let $G^*$ denote the graph given by the union of these components, all disconnected from each other.  Let $\psi_i=\psi_{y_i}$ denote the map taking $\bar J_{y_i}$ into $\block_i$.  By construction it maps $J_{y_i}$ into the corresponding inner block $B_i$.  Denote $\Lambda=\cup_i J_{y_i}$ and $\Lambda^*=\cup_i B_i$.

Let $X_t^*$ denote the Glauber dynamics on $G^*$, inheriting the boundary conditions of the blocks $\block_i$, and let $\mu^*$ denote its stationary distribution.  As in the upper bound we couple the dynamics on $G$ and $G^*$ so that if $x$ is in some $\bar J_{y_i}$ then we couple the updates of $x$ and $\psi_i(x)$ and for all other vertices in $\Z_n^d$ we take the updates to be independent.  We will assume our initial conditions satisfy $X_0(x)=X^*_0(\psi_i(x))$  for all $i$ and $x\in \bar J_{y_i}$.

As in the case of the upper bound, with minor modifications of Lemma~\ref{lem-sparse-product} and Lemma~\ref{lem-xt-bound-aw} we have that for $t\leq 2 \sobinf^{-1}d\log n$,
\[
\max_{x_0} \left \| \P_{x_0}(X_t \in \cdot) - \mu \right \|_{\tv} \geq \max_{x_0^*}  \left \| \P_{x_0^*}(X_t^*(\Lambda^*) \in \cdot) - \mu^*_{\Lambda^*} \right \|_{\tv} - 2n^{-2d/3} \,.
\]
It is sufficient to consider this range of $t$ since a minor modification of the proof of equation~\eqref{e:tmixUpperBound} implies that for all $\delta>0$ and large enough $n(\delta)$ we have that $\tmix(\delta) \leq  2\sobinf^{-1}d\log n$.  Now, similarly to equation~\eqref{e:LInftyBoundBlock}, when $t$ is of order $\log n$ then
\[
\left\|\P_{x_0}(X_t^*(B_i)\in\cdot) - \mu^*|_{B_i} \right\|_{L^\infty(\pi_i|_{B_i})}
=o(1).
\]
Hence applying Proposition~\ref{prop-l1-l2} we have that
\begin{align}\label{e:lowerBoundPlusFinal}
&\max_{x_0} \left \| \P_{x_0}(X_t \in \cdot) - \mu \right \|_{\tv} \nonumber\\
&\qquad\geq
\Psi\bigg(\sum_{i=1}^L \max_{x_0}\left\|\P_{x_0}(X_t^*(B_i)\in\cdot) - \mu^*|_{B_i} \right\|_{L^2(\mu^*)}^2\bigg)-o(1) \nonumber\\
&\qquad\geq \Psi\bigg(\sum_{j=0}^d 3^{-d} \lltwo^{(j)}_t n^{d-j} \log^{-3d}n\bigg)-o(1)\,,
\end{align}
where $\Psi(x) = 2\Phi(\frac12\sqrt{x})-1$.

\subsection{Existence of Cutoff}

We are now able to prove establish the existence of cutoff for $X_t$.  Define the approximate mixing time as
\[
t^\star = \inf\left\{t > 0 \;:\; \sum_{j=0}^d 3^{-d} \lltwo^{(j)}_t n^{d-j} \log^{-3d}n \leq \log n \right\}\,.
\]
Let $\lambda_j(m)$ denote the second eigenvalue of the dynamics on $\mathcal{T}_j$.  Then $\sobinf^{-1}\leq \lambda_j(m) \leq 1$ and so by Theorem~\ref{thm-l2-sobolev} we have that
\begin{align}
  \label{eq-mjt-ubound}
  \lltwo^{(j)}_t &= \max_{x_0} \left\| \P_{x_0}\big(X^{*,j}_t(B_j) \in \cdot \big)  - \mu^{*,j}_{B_j} \right\|^2_{L^2(\mu^*_{B_j})}\nonumber\\
  &\leq \exp(-\lambda_j(m)t + O(\sobinf^{-1}\log\log n)) \, ,
\end{align}
and so $t^\star \leq (d+o(1)) \sobinf^{-1}\log n$.  Hence we can apply equation~\eqref{e:upperBoundPlusFinal} to show that
\begin{equation*}
\max_{x_0} \left\|\P_{x_0}(X_{t^\star+s}\in\cdot)-\mu \right\|_\tv \leq
\sqrt{d 6^d  \log^{1-4d}n} + 2n^{-2d/3}=o(1)\, .
\end{equation*}
while by equation~\eqref{e:lowerBoundPlusFinal} we have that
\begin{equation*}
\max_{x_0} \left \| \P_{x_0}(X_{t^\star} \in \cdot) - \mu \right \|_{\tv} \geq \Psi\left(\log n\right)-o(1)=1-o(1)\,,
\end{equation*}
where $\Psi(x) = 2\Phi(\frac12\sqrt{x})-1$.  It follows that $\tmix(\epsilon)$ is between $t^\star$ and $t^\star+s$ for any $0<\epsilon<1$.
Recalling that results of \cite{HaSi} imply that the $L^1$ mixing time of the Glauber dynamics for the Ising model on $(\Z/n\Z)^d$ has order at least $\log n$, we have that $t^\star$ is also of order at least $\log n$ and hence $X_t$ has cutoff at $t^\star$ with a window of width at most $s=100 d^2 \sobinf^{-2}\log \log n$.

\subsection{Cutoff Location}

The rest of the section concerns relating the quantities $\lltwo^{(j)}_t$ to eigenvalues of infinite volume Glauber dynamics to estimate $t^\star$ and the cutoff location.


\begin{lemma}\label{lem-L2-bounds}
There exist $\hat{\lambda}_j$ such that $\lambda_j(m)\to \hat{\lambda}_j$ as $m\to\infty$.  Moreover, there exists $c=c(\beta,j,d)>0$ such that for $m = \log^3 n$,
\[
|\lambda_j(m) - \hat{\lambda}_j| \leq \frac{c\log\log n}{\log n}
\]
and for $0 \leq t \leq 2d\sobinf^{-1}\log n$ and $n$ sufficiently large,
\begin{align}\label{eq-L2-bounds}
 \mathrm{e}^{-\hat{\lambda}_j t-c\log\log n} - 32 n^{-4d/3} \leq \lltwo^{(j)}_t(m) \leq  \mathrm{e}^{-\hat{\lambda}_j t+c\log\log n}\,.
\end{align}
\end{lemma}
\begin{proof}

In the proof of the upper bound in Section~\ref{sec:plusUBound} we could apply the same analysis to the Markov chain $X^{*,j}_t$ on $\block^{(j)}$ but crucially mapping components of the update support into blocks $\block_i$ of side-length $\log^3 n$ (rather than $\log^3 m$).  By construction, all components of the update support will be of the form $\block^{(j)}$.  With $s=100 d^2 \sobinf^{-2}\log \log n$, an adaptation of equation~\eqref{e:upperBoundPlusFinal} gives us that
\begin{equation}\label{e:apdaptedUBoundPlus}
\max_{x_0^*} \left\|\P_{x_0^*}(X^{*,j}_{t+s}\in\cdot)-\mu^{*,j} \right\|_\tv \leq
\sqrt{ 2^d   \lltwo^{(j)}_t \log^{-7d}n} + 2n^{-2d/3}\, ,
\end{equation}
when $t\leq 2d \sobinf^{-1}\log n$.  We can also vary the size of the blocks that the components of the update support are mapped into for instance taking them of side-length $m'\in[\frac45 m, \frac65 m]$ in which case we get the bound
\begin{equation}\label{e:generalEmbedBoundPlus}
\max_{x_0^*} \left\|\P_{x_0^*}(X^{*,j}_{t+s}\in\cdot)-\mu^{*,j} \right\|_\tv \leq
\sqrt{ 2^d   \lltwo^{(j)}_t(m') \log^{-7d}n} + 2n^{-2d/3}\, .
\end{equation}

Let  $\Omega^*_m$ denote the  state space of $X^{*,j}_t$. Since $\log\log (1/\mu^{*,j}(\sigma)) \geq (3d+o(1))\log \log n$ for all $\sigma \in \Omega^*_m$ and since $\lambda_j(m) \leq 1$ (vertices are updated at rate 1), Theorem~\ref{thm-l2-sobolev}
implies that for large $n$
\begin{align}\label{e:generalLTwoBound}
\lltwo^{(j)}_t(m') &\leq  \exp\left( 1-\lambda_j(m')\left(t-\frac1{4 \sobinf}\log\log \left(1/\mu^{*,j}(\sigma)\right)\right)\right) \nonumber\\
&\leq  \mathrm{e}^{-\lambda_j(m') t+\frac{3d+o(1)}{4\sobinf}\log\log n} \, .
\end{align}

Now a standard lower bound on the total variation distance in terms of the spectral gap (cf.\ its discrete-time analogue \cite{LPW}*{equation (12.13)}) gives that for all $t>0$,
\begin{equation}\label{e:gapLBound}
\mathrm{e}^{-\lambda(m) t} \leq 2 \max_{x_0^*} \left \| \P(X^{*,j}_{t} \in\cdot)- \mu^{*,j}\right \|_\tv \,.
\end{equation}

Combining equations~\eqref{e:generalLTwoBound}, \eqref{e:generalEmbedBoundPlus} and~\eqref{e:gapLBound} we have that for any $0 \leq t \leq \frac{3}{5\lambda_j(m')}\log n$, $s=100 d^2 \sobinf^{-2}\log \log n$ and $m'\in[\frac45 m, \frac65 m]$,
\begin{align}\label{e:combinedLambdaBound}
\mathrm{e}^{-\lambda_j(m) (t+s)} &\leq 2 \max_{x_0^*} \left \| \P(X^{*,j}_{t+s} \in\cdot)- \mu^{*,j}\right \|_\tv \nonumber\\
&\leq 2\sqrt{ 2^d   \lltwo^{(j)}_t(m') \log^{-7d}n} + 4n^{-2d/3} \nonumber\\
&\leq  \mathrm{e}^{-\lambda_j(m') t+\frac{3d+o(1)}{4\sobinf}\log\log n} \, .
\end{align}
Since $\lambda_j(m')\leq 1$ for all $m'$ we may take $t=\frac12\log n = \frac12m^{1/3}$ and get that,
\[
 \left[\frac12m^{1/3}+100 d^2 \sobinf^{-2}\log m\right] \lambda_j(m) \geq  \left[\frac12m^{1/3}\right] \lambda_j(m')  - \frac{d+o(1)}{4\sobinf}\log m \, ,
\]
and hence with $c_1=\frac{2d+o(1)}{4\sobinf}+200 d^2 \sobinf^{-2}$
\[
\lambda_j(m) -\lambda_j(m') \geq  - (c_1 +o(1))\frac{\log m}{m^{1/3}} \,.
\]
Rearranging the role of $m$ and $m'$ we have that for all $m'\in[\frac56 m, \frac54 m]$ then
\[
\lambda_j(m) -\lambda_j(m') \leq  (c_1 +o(1))\frac{\log m'}{(m')^{1/3}} \leq (2c_1+o(1))\frac{\log m}{m^{1/3}} \,.
\]
Combining the previous two equations we have that for all $m'\in[\frac56 m, \frac65 m]$
\begin{equation}\label{e:lambdaDifference}
|\lambda_j(m) -\lambda_j(m')| \leq (2c_1+o(1))\frac{\log m}{m^{1/3}} \,.
\end{equation}
This implies that $\lambda_j(m)$ converges to some limit $\hat{\lambda}_j$ as $m\to\infty$ and that
\begin{align}\label{e:lambdaConvergenceBound}
|\lambda_j(m) - \hat{\lambda}_j| &\leq \sum_{k=1}^\infty |\lambda_j((\tfrac65)^{k-1}m) -\lambda_j((\tfrac65)^{k}m)| \nonumber\\
&\leq \sum_{k=1}^\infty (2c_1+o(1))\frac{\log ((\tfrac65)^{k}m)}{(\tfrac65)^{k/3}m^{1/3}}\nonumber\\
&\leq \frac{35 c_1\log m}{m^{1/3}}=\frac{105c_1  \log\log n}{\log n}
\end{align}
for large $m$.

Substituting~\eqref{e:combinedLambdaBound} into~\eqref{e:lambdaConvergenceBound} and using the fact that $x\leq y+z$ implies $2y^2\geq x^2-2z^2$ we have that for a large enough constant $c$ and large $n$,
\begin{align*}
\mathrm{e}^{-2\hat{\lambda}_j t -c\log\log n} - 32n^{-4d/3} \leq \lltwo^{(j)}_t(m') \leq \mathrm{e}^{-2\hat{\lambda}_j t +c\log\log n}\,
\end{align*}
which completes the lemma.
\end{proof}

Lemma~\ref{lem-L2-bounds} implies that
\[
\left|t^\star - \Big(\min_{0\leq j \leq d-1} (d-j)\hat{\lambda}_j^{-1}\Big)\log n\right| \leq c\log \log n\,.
\]
The last step in the proof of Theorem~\ref{t:plusGeneral} will be to show that $\hat{\lambda}_j=\lambda_j^{(\infty)}$.

Let $Y^{(j)}_t$ denote the Glauber dynamic on $\HS_j$.
Consider the function $f_j(\sigma) = \sum_{v\in\HS_j}\frac{\sigma(v)-\E\sigma(v)}{|v|^{d+1}}$ and let
\[
\xi_{t}^{(j)} = \E_+ f_j(Y^{(j)}_t) - \E_- f_j(Y^{(j)}_t) \, .
\]
By the characterization of the spectral gap as the slowest exponential rate of decay to 0 of $\E_{y_0} f(Y_t)$ for mean zero $L^2$ functions under the action of the semi-group we have that
\[
\exp(-\lambda_j^{(\infty)}) \geq \limsup_{t\to\infty} (\xi_t^{(j)})^{1/t} \, .
\]
As before let $X^{*,j}_t$ denote the Glauber dynamics on the block $\block^{(j)}$ of side-length $m=\log^3 n$.  Let
\[
B_m=\left\{(x_1,\ldots,x_d)\in \Z^d:0\leq x_i \leq m/2\right\} \, ,
\]
which we will associate with subsets both of $\block^{(j)}=\Z^d_m$ and $\HS_j$ in the natural way.  By another disagreement percolation argument of the type in Lemma~\ref{lem-barrier-couple} we can couple $Y^{(j)}_t$ and $X^{*,j}_t$ starting from the all plus configuration so that
\begin{equation}\label{e:magnetizationCoupling}
\P_+(Y^{(j)}_{\log n} (B_m)\neq X^{*,j}_{\log n}(B_m)) \leq \exp(-\log^2 n)
\end{equation}
and similarly for all minus initial configurations.  By the standard monotone coupling and symmetry of $\block^{(j)}$
\begin{align}\label{e:monocouplingBound}
\max_{x_0^*} \left\|\P_{x_0^*}(X^{*,j}_{t}\in\cdot)-\mu^{*,j} \right\|_\tv &\leq \sum_{v\in\block^{(j)}} \Big( \E_+ X^{*,j}_{t} -  \E_- X^{*,j}_{t} \Big) \nonumber \\ &\leq 2^d \sum_{v\in B_m} \Big( \E_+ X^{*,j}_{t} -  \E_- X^{*,j}_{t} \Big) \, .
\end{align}
By the coupling we have that of $Y^{(j)}_t$ and $X^{*,j}_t$ it follows that,
\begin{align}\label{e:magnetizationCouplingA}
\sum_{v\in B_m} \Big( \E_+ X^{*,j}_{\log n} -  \E_- X^{*,j}_{\log n} \Big) &\leq \sum_{v\in B_m} \Big( \E_+ Y^{(j)}_{\log n} -  \E_- Y^{(j)}_{\log n} \Big)+ m^3\mathrm{e}^{-\log^2 n} \nonumber\\
&\leq m^{d+1}\xi_{\log n}^{(j)} + m^3\mathrm{e}^{-\log^2 n}.
\end{align}
The standard lower bound on the total variation distance in terms of the spectral gap (cf.\ its discrete-time analogue \cite{LPW}*{equation (12.13)}) gives that,
\begin{equation}\label{e:gapLBoundA}
\mathrm{e}^{-\lambda_j(m) \log n} \leq 2 \max_{x_0^*} \left \| \P(X^{*,j}_{\log n} \in\cdot)- \mu^{*,j}\right \|_\tv \,.
\end{equation}
and so combining equations~\eqref{e:monocouplingBound}, \eqref{e:magnetizationCouplingA} and~\ref{e:gapLBoundA} we have that
\[
\mathrm{e}^{-\hat{\lambda}_j}  \leq \limsup_n \left((2m)^{d+1}\xi_{\log n}^{(j)}\right)^{1/\log n} + \left(m^3\mathrm{e}^{-\log^2 n}\right)^{1/\log n} \leq \mathrm{e}^{-\lambda_j^{(\infty)}} \, ,
\]
and hence $\hat{\lambda}_j\geq \lambda_j^{(\infty)}$.

It remains to prove that $\hat{\lambda}_j\leq \lambda_j^{(\infty)}$. Now fix $\epsilon>0$ and recall the Dirichlet form \eqref{eq-dirichlet-form}, according to which
\[
\lambda_j^{(\infty)} = \inf_{f\in L^2(\{\pm1\}^{\HS_j},\mu_j)}  \frac{\mathscr{E}_{\mu_j}(f,f)}{\var_{\mu_\infty}(f)}\,,
\]
where $\mu_j$ is the stationary measure of the infinite-volume Ising model.
For any $f\in L^2(\{\pm1\}^{\HS_j},\mu_j)$
with $\mathscr{E}_{\mu_j}(f,f)<\infty$ we can find a sequence of functions $f_n\in L^2(\{\pm1\}^{\HS_j},\mu_j)$
each of which depends only on a finite number of spins such that $f_n\to f$ in $L^2(\{\pm1\}^{\HS_j},\mu_j)$ and
$\mathscr{E}_{\mu_j}(f_n,f_n)\to\mathscr{E}_{\mu_j}(f,f)$ (see e.g.\ the proof of \cite{Liggett}*{Lemma 4.3}).
So take $g\in L^2(\{\pm1\}^{\HS_j},\mu_j)$  depending
only on a finite number of spins such that
\[
\frac{\mathscr{E}_{\mu_j}(g,g)}{\var_{\mu_j}(g)} \leq \frac{\mathscr{E}_{\mu_j}(f,f)}{\var_{\mu_j}(f)} +\epsilon\,.
\]
For some large enough $M$ we have that $g$ is a function of the spins in the box $\Delta_M=\{x\in \HS_j: \|x\|_\infty\leq M\}$.
We compare the Ising model on $\HS_j$ and on $\block^{(j)}$ with side-length $m$, identifying the vertices $\Delta_M\subset \HS_j$ with those in $\Delta_M^*=\{x\in \block^{(j)}: x \hbox{ mod } m \in \Delta_M\}$.
By the strong spatial mixing property,
\[
\|\mu_j|_{\Delta_M} - \mu^{*,j}_m|_{\Delta_M^*}\|_{L^\infty} \to 0
\]
as $m\to\infty$ and hence
\[
\frac{\mathscr{E}_{\mu_j}(g,g)}{\var_{\mu_j}(g)} = \lim_{m\to\infty} \frac{\mathscr{E}_{\mu^{*,j}_m}(g,g)}{\var_{\mu^{*,j}_m}(g)} \geq \lim_{m\to\infty} \lambda_j(m) =  \hat{\lambda}_j \,,
\]
where the inequality follows from the characterization of the spectral gap by the Dirichlet form.  This implies that
\[
\lambda_j^{(\infty)} = \inf_f \frac{\mathscr{E}_{\mu_j}(f,f)}{\var_{\mu_j}(f)} \geq \hat{\lambda}_j -\epsilon\,,
\]
and so $\lambda_j^{(\infty)}\geq \hat{\lambda}_j$ which completes the proof of Theorem~\ref{t:plusGeneral}.

\medskip
\section*{Acknowledgments}
We are grateful to Pietro Caputo, Fabio Martinelli, Fabio Toninelli and Yuval Peres for useful discussions.

\end{document}